\documentclass[12pt]{amsart}
\usepackage{amsmath}
\usepackage{amssymb}
\usepackage{caption}
\usepackage{subcaption}
\usepackage[curve]{xypic}
\usepackage{cite}
\usepackage{enumitem}
\usepackage{color}
\usepackage{url}
\usepackage[usenames,dvipsnames]{xcolor}
\usepackage{graphicx}
\usepackage{verbatim}
\usepackage{tikz}
\usepackage{tikz-cd}
\usepackage{float}

\makeatletter 
\def\@cite#1#2{{\m@th\upshape\bfseries%
[{#1\if@tempswa{\m@th\upshape\mdseries, #2}\fi}]}}
\makeatother 

\newtheorem{cor}{Corollary}{\bf}{\it}
{\bf}{\it}
{\bf}{\it}
{\bf}{\it}
\newtheorem{mthm}{Main Theorem}{\bf}{\it}
{\bf}{\it}
\newtheorem{remark}{Remark}{\bf}{\it}

\newtheorem{ssthm}[subsection]{Theorem}{\bf}{\it}
\newtheorem{sslem}[subsection]{Lemma}{\bf}{\it}
\newtheorem{sscoro}[subsection]{Corollary}{\bf}{\it}
\newtheorem{ssprop}[subsection]{Proposition}{\bf}{\it}

\newtheorem*{thm}{Theorem}{\bf}{\it}
\newtheorem*{lem}{Lemma}{\bf}{\it}
\newtheorem*{sublem}{Sublemma}{\bf}{\it}
\newtheorem*{rem}{Remark}{\bf}{\it}

\newcommand{\N}{\mathcal N}

\newcommand{\Q}{\mathbb Q}
\newcommand{\R}{\mathbb R}
\newcommand{\C}{\mathbb C}
\newcommand{\p}{\mathbb P}
\newcommand{\ds}{\displaystyle}

\newcommand{\ra}{\longrightarrow}

\newcommand{\M}{\mathcal{M}}
\newcommand{\hyp}{\mathcal{H}^{hyp}}
\newcommand{\strata}{\mathcal{H}}
\newcommand{\Cyl}{\mathcal C}
\newcommand{\D}{\mathcal D}
\newcommand{\V}{\mathcal V}
\newcommand{\W}{\mathcal W}
\newcommand{\rel}{\mathrm{rel}}
\newcommand{\wt}{\widetilde}
\newcommand{\Mod}{\mathrm{Mod}}

\newcommand{\GL}{\mathrm{GL} }

\newcommand{\Pres}{\mathrm{Pres}_{(X, \omega)} }
\newcommand{\Twist}{\mathrm{Twist}_{(X, \omega)}}
\newcommand{\ackn}{\noindent{\sc Acknowledgement }\hspace{5pt} }

\newcommand{\cC}{\mathcal C}

\newcommand{\cV}{\mathcal V}
\newcommand{\cS}{\mathcal S}
\newcommand{\cW}{\mathcal W}
\newcommand{\T}{\mathcal T}

\begin{document}

\title[$\GL_2 \R$ Orbit Closures in Hyperelliptic Components]{$\GL_2 \R$ Orbit Closures in Hyperelliptic Components of Strata}

\author[Apisa]{Paul Apisa}
\address{Department of Mathematics, University of Chicago, Chicago, IL, 60615} \email{paul.apisa@gmail.com}

\maketitle 

\begin{abstract}
The object of this paper is to study $\GL_2 \R$ orbit closures in hyperelliptic components of strata of abelian differentials. The main result is that all higher rank affine invariant submanifolds in hyperelliptic components are branched covering constructions, i.e. every translation surface in the affine invariant submanifold covers a translation surface in a lower genus hyperelliptic component of a stratum of abelian differentials. This result implies finiteness of algebraically primitive Teichmuller curves in all hyperelliptic components for genus greater than two. A classification of all $\GL_2 \R$ orbit closures in hyperelliptic components of strata (up to computing connected components and up to finitely many nonarithmetic rank one orbit closures) is provided. Our main theorem resolves a pair of conjectures of Mirzakhani in the case of hyperelliptic components of moduli space.
\end{abstract}

\section{Introduction}

Let $\M_g$ be the moduli space of genus $g$ Riemann surfaces and let $\mathcal{H}_g$ be the sublocus of hyperelliptic Riemann surfaces. The inclusion of $\mathcal{H}_g$ into $\M_g$ is totally geodesic with respect to the Kobayashi metrics. Teichm\"uller geodesic flow and complex scalar multiplication generate a $\GL_2(\R)$ action on $\Omega \M_g$ - the moduli space of holomorphic one-forms on closed genus $g$ Riemann surfaces. Since $\mathcal{H}_g$ is totally geodesically embedded in $\M_g$, the collection $\Omega \mathcal{H}_g$ of holomorphic one-forms on genus $g$ hyperelliptic Riemann surfaces is invariant under the $\GL_2(\R)$ action.

Every genus $g$ hyperelliptic Riemann surface may be written as the normalization of the projective curve defined in affine coordinates by the equation $y^2 = f(x)$ where $f$ is a polynomial of degree $2g+1$ or $2g+2$ with simple roots. The hyperelliptic Riemann surface $X$ defined by the equation $y^2 = f(x)$ admits a holomorphic one-form $\omega = \frac{dx}{y}$. The zeros of this one-form occur precisely at the points of the curve which are at infinity. This set contains either a single point fixed by the hyperelliptic involution or two points exchanged by the involution. The question we undertake in the sequel is - ``what is the $\GL_2 \R$ orbit closure of this one-form?" Call this orbit closure $\M$. Every such orbit closure is a complex subvariety of $\Omega \M_g$ by work of McMullen~\cite{Mc5} in genus two and work of Eskin-Mirzakhani~\cite{EM}, Eskin-Mirzakhani-Mohammadi~\cite{EMM}, and Filip~\cite{Fi1} in general. Surprisingly, we find
\begin{mthm}\label{M1}
Let $\M$ be as in the preceding paragraph. Let $\pi: \Omega \M_g \ra \M_g$ be the forgetful map that sends a translation surface to its underlying Riemann surface. If $g > 2$, then there is a finite union $C_g \subseteq \Omega \M_g$ of subvarieties  of complex dimension at most three depending only on $g$, so that at least one of the following three possibilities occurs:
\begin{enumerate}
\item $\M$ is contained in $C_g$
\item $\pi(\M)$ coincides with the hyperelliptic locus
\item $\pi(\M)$ is a locus of branched covers contained in the hyperelliptic locus.
\end{enumerate}
\end{mthm}

Though it may not seem so, complex dimension three is not arbitrary. The reason it forms a natural threshold is rank. The tangent bundle of an orbit closure $\M$ at a point $(X, \omega)$ is naturally identified with a complex linear subspace of $H^1(X, Z(\omega); \C)$ where $Z(\omega)$ is the zero set of $\omega$, see~\cite{Wcyl} for details. If $p$ is the projection from relative to absolute cohomology, then Avila-Eskin-M\"oller~\cite{AEM} showed that $p\left( T_{(X, \omega)} \M \right)$ is a complex symplectic vector space. The rank of an affine invariant submanifold is defined to be half the complex dimension of the vector space $p\left( T_{(X, \omega)} \M \right)$. An orbit closure is said to be higher rank if its rank is larger than one. In the cases that we consider an orbit closure is low rank if and only if its complex dimension is three or less. 

The notion of rank allows us to recast the main question into the language of Teichm\"uller dynamics. The space $\Omega \M_g$ admits a $\GL_2 \R$-invariant stratification by prescribing the number and degree of vanishing of the zeros of the holomorphic one-forms. Given a partition $\kappa$ of $2g-2$, $\Omega \M_g(\kappa)$ denotes the stratum of holomorphic one-forms with $|\kappa|$ zeros whose degrees of vanishing form the set $\kappa$. The stratification on $\Omega \M_g$ induces a $\GL_2 \R$-invariant stratification on $\Omega \strata_g$. In this work, since we are interested in the orbit closures of abelian differentials of the form $\frac{dx}{y}$ we are interested in components of strata with either one zero or two zeros exchanged by the hyperelliptic involution. These are denoted by $\hyp(2g-2)$ and $\hyp(g-1,g-1)$ respectively. The main question then becomes ``what are the higher rank orbit closures in $\hyp(2g-2)$ and $\hyp(g-1,g-1)$?"

\begin{remark}
The notation change from $\Omega \strata_g$ to $\hyp$ is to distinguish specific components of $\Omega \strata_g$. While $\Omega \strata_g(2g-2)$ is connected and equal to $\hyp(2g-2)$, $\Omega \strata_g(g-1,g-1)$ is disconnected outside of genus two. For $g > 2$, $\Omega \strata_g(g-1,g-1)$ contains two components - $\hyp(g-1,g-1)$ - which contains abelian differentials with two zeros exchanged by the hyperelliptic involution and a second component that contains abelian differentials whose zeros are both Weierstrass points. In this work, we restrict our attention to $\hyp(2g-2) = \Omega \strata_g(2g-2)$ and to the component $\hyp(g-1,g-1)$ of $\Omega \strata_g(g-1,g-1)$.
\end{remark}

Posed in this way, the question already has a conjectural answer. Mirzakhani conjectured that higher rank orbit closures have field of definition $\Q$ and that they are branched covering constructions. These conjectures were first published in~\cite{Wcyl}. Recent work of McMullen, Mukamel, and Wright~\cite{MMW} provides a counterexample to this conjecture in a nonhyperelliptic component of a stratum in genus three. It remains an open question to determine in which components of which strata these conjectures hold. An orbit closure $\M$ is said to be branched covering construction if for every abelian differential $(X, \omega)$ in $\M$ - where $X$ is a Riemann surface and $\omega$ a holomorphic one-form - there is a quadratic differential $(Y, q)$ on a Riemann surface $Y$ and a holomorphic map $f: X \ra Y$ such that $\omega^2 = f^*q$. In the case of the hyperelliptic locus, which is automatically a stratum of branched covers of quadratic differentials on punctured spheres, the definition is modified to mean a locus of pullbacks of abelian differentials on lower genus Riemann surfaces.

%

The main result of this work, which implies Theorem~\ref{M1}, is that the Mirzakhani conjectures hold in hyperelliptic components of strata.

%
%
\begin{mthm}\label{T:M-1}
Let $\M$ be an orbit closure in $\hyp(2g-2)$ or $\hyp(g-1,g-1)$ of complex dimension at least four. If $\dim \M = 2r$ then $\M$ is a branched covering construction over $\hyp(2r-2)$ and if $\dim \M = 2r+1$ then $\M$ is a branched covering construction over $\hyp(r-1,r-1)$. The covers are branched over the zeros of the holomorphic one-forms and commute with the hyperelliptic involution.
\end{mthm}

The letter $r$ was chosen in the above theorem statement since it coincides with the rank of $\M$. 

\begin{remark}
It is important to note that this result is about hyperelliptic components of strata and not about hyperelliptic loci of abelian differentials. An open problem related to this work is to find a classification of the orbit closures in other strata of the hyperelliptic locus, i.e. to analyze the $\GL_2 \R$ dynamics of $\Omega \strata_g(\kappa)$ for $\kappa$ beyond $(2g-2)$ and $(g-1,g-1)$. 
\end{remark}

As a corollary of Theorem~\ref{T:M-1} we have finiteness of algebraically primitive Teichm\"uller curves in hyperelliptic components. In the case of $\hyp(g-1,g-1)$ this result was the main result of M\"oller~\cite{M3}.

\begin{mthm}
In $\hyp(2g-2)$ and $\hyp(g-1,g-1)$ there are finitely many algebraically primitive Teichm\"uller curves for $g>2$.
\end{mthm}
\begin{proof}
Suppose not to a contradiction. Let $C_i$ be an infinite sequence of distinct algebraically primitive Teichm\"uller curves. By Eskin-Mirzakhani~\cite{EM} a subsequence equidistributes in a finite union of connected affine invariant submanifolds $\M = \bigcup_i \M_i$. By Matheus-Wright~\cite{MW} algebraically primitive Teichm\"uller curves cannot equidistribute in the connected component of any stratum when $g>2$. Main Theorem~\ref{M1} implies that no $\M_i$ is higher rank since this would imply that $C_i$ is not geometrically primitive (and hence not algebraically primitive). Finally, no $\M_i$ is rank one since these orbit closures only contain finitely many nonarithmetic Teichm\"uller curves by Lanneau-Nguyen-Wright~\cite{LNW}. Therefore, we have a contradiction.
 \end{proof}

Work of Eskin, Filip, and Wright~\cite{EFW} establishes the following.

\begin{thm}[Eskin, Filip, Wright~\cite{EFW} Theorem 1.5]
Any infinite collection of nonarithmetic rank one $\GL_2 \R$ orbit closures admits a subsequence that equidistributes in a rank two affine invariant submanifold.
\end{thm}

It follows immediately from Main Theorem~\ref{M1} that

\begin{mthm}
For $g>2$, all but finitely many orbit closures in $\hyp(2g-2)$ and $\hyp(g-1,g-1)$ are branched covering constructions, and each of them has dimension at most three. 
\end{mthm}
\begin{proof}
Suppose to a contradiction that there is an infinite sequence $C_i$ of orbit closures in $\hyp(2g-2)$ or $\hyp(g-1,g-1)$ that are not branched covering constructions. By Main Theorem~\ref{M1} these orbit closures are rank one. Since arithmetic rank one orbit closures are torus covers, each $C_i$ is nonarithmetic rank one. Therefore, Eskin, Filip, and Wright~\cite{EFW} (Theorem 1.5) implies that the sequence equidistributes in a union of rank two affine invariant submanifolds. By Main Theorem~\ref{M1} these rank two orbit closures are branched covering constructions and therefore so are the $C_i$, which is a contradiction. 
\end{proof}

\noindent This theorem implies that outside of a subvariety of dimension at most three, we understand the closure of every complex geodesic in $\hyp(2g-2)$ and $\hyp(g-1,g-1)$. However, determining this subvariety remains open in all genera greater than two.  

The results in this paper represent the first such classification of orbit closures that holds in all genera. The argument is based on a degeneration argument that takes advantage of recent work of Mirzakhani and Wright~\cite{MirWri} on partial compactifications of affine invariant submanifolds. This is the first time that such an argument has been used to study orbit closures.

\subsection{Classification of Orbit Closures} 

We now offer a coarse classification of orbit closures in hyperelliptic components of strata. By coarse classification we mean a classification of orbit closures up to finitely many nonarithmetic closed orbits and up to classifying connected components of orbit closures.

\begin{mthm}[Classification of Orbit Closures in $\hyp(2g-2)$]
The affine invariant submanifolds in $\hyp(2g-2)$ are:
\begin{enumerate}
\item Countably many Teichm\"uller curves which arise from torus covers branched over one point.
\item (If $g \equiv 2 \text{ mod } 3$) Countably many Teichm\"uller curves whose trace fields are degree two and which equidistribute in an affine invariant submanifold of $\hyp(2g-2)$ arising from a branched cover construction over $\mathcal{H}(2)$.
\item Finitely many Teichm\"uller curves beyond the previous two families.
\item Finitely many rank $r>1$ affine invariant submanifolds for each $2r-1 \mid 2g-1$. These are branched covering constructions over $\hyp(2r-2)$.
\end{enumerate}
\end{mthm}

\begin{cor}[Characterization of Optimal Dynamics in $\hyp(2g-2)$]
Every orbit in $\hyp(2g-2)$ is either closed or equidistributed if and only if $2g-1$ is prime.
\end{cor}
\begin{proof}
Higher rank proper orbit closures arise if and only if $2r-1$ divides $2g-1$ for some $r>1$.
 \end{proof}

In the case of $g=2$ this characterization of optimal dynamics follows from McMullen's classification of orbit closures in genus two, see~\cite{Mc}, ~\cite{Mc4}, and ~\cite{Mc5}. In the case of $g=3$ this theorem was the main theorem of Nguyen-Wright~\cite{NW}.

\begin{mthm}[Classification of Orbit Closures in $\hyp(g-1,g-1)$]
The affine invariant submanifolds in $\hyp(g-1,g-1)$ are:
\begin{enumerate}
\item Countably many Teichm\"uller curves which arise from torus covers branched over one point.
\item (If $g \equiv 0 \text{ mod } 3$) Countably many Teichm\"uller curves whose trace fields are degree two and which equidistribute in affine invariant submanifolds that are branched covering constructions over $\strata(2)$.
\item Finitely many Teichm\"uller curves beyond the previous two families.
\item Countably many orbit closures that are branched covering constructions over $\strata(0,0)$.
\item (If $g$ is even) Countably many orbit closures that cover genus two eigenform loci; these equidistribute in affine invariant submanifolds that are branched covering constructions over $\strata(1,1)$.
\item Finitely many three dimensional orbit closures beyond the previous two families. These are necessarily rank one and nonarithmetic. 
\item Finitely many rank $r>1$ affine invariant submanifolds for each $r \mid g$. These are branched covering constructions of $\hyp(r-1,r-1)$.
\item Finitely many rank $r>1$ affine invariant submanifolds for each $2r-1 \mid g$. These are branched covering constructions of $\hyp(2r-2)$.
\end{enumerate}
\end{mthm}

In the case of $g=2$ this theorem follows from McMullen's classification of orbit closures in genus two. In the case of $g=3$ this theorem was one of the main theorems of Aulicino-Nguyen~\cite{AN}.

\begin{cor}
There are no higher rank proper even dimensional orbit closures in $\hyp(g-1,g-1)$ if and only if $g = 2^n$ for some $n$. There are no higher rank proper odd dimensional orbit closures in $\hyp(g-1,g-1)$ if and only if $g$ is prime.
\end{cor}

\subsection{Relation to Previous Results}

The origin of this work begins in the study of orbit closures in genus two. In the early 2000s, Calta~\cite{Ca} and McMullen~\cite{Mc} discovered an infinite family of closed complex geodesics in $\M_2$ that projected to curves $W_D$ on Hilbert modular surfaces $X_D$. In a subsequent series of papers~\cite{McM:spin}, ~\cite{Mc4} and~\cite{Mc5}, McMullen showed that the holomorphic one-forms on genus two Riemann surfaces whose $\GL_2(\R)$ are not dense are exactly the eigenforms of real multiplication. Bainbridge determined the Euler characteristic of each Weierstrass curve $W_D$ and the Lyapunov exponents of the Kontsevich-Zorich cocycle restricted to $W_D$ in~\cite{Ba}. Mukamel~\cite{Muk14} determined the orbifold points and homeomorphism type of each $W_D$. These results represent, up to classifying loci of torus covers, a classification of orbit closures in $\hyp(2)$ and $\hyp(1,1)$.

After McMullen's classification in genus two, attention turned to periodic points in genus two. M\"oller~\cite{M2} established that the only periodic points on Veech surfaces in $\Omega \M_2$ are Weierstrass points. The result was then extended to generic surfaces in all components of strata in Apisa~\cite{Apisa-mp1}. 

For more general results, see Eskin-Mirzakhani~\cite{EM}, Eskin-Mirzakhani-Mohammadi~\cite{EMM}, and Filip~\cite{Fi1}. We suspect that orbit closures in hyperelliptic loci could be productively studied by exploring the connection between the braid group and Hodge structure, see for instance McMullen~\cite{McM:braid}.

\subsection{Organization of the Paper and Remarks on the Proof}

In Section~\ref{S:Background} we provide background on results that will be fundamental for the proofs of our results. In Section~\ref{LindseyTrees} we will discuss a combinatorial model - the Lindsey half-tree - created by Kathryn Lindsey in~\cite{Lindsey} to study horizontally periodic translation surfaces in hyperelliptic components of strata. This model explains much of why orbit closures are particularly well-behaved in hyperelliptic components of strata. In Section~\ref{BCCs} we will define branched covering constructions rigorously and devise a criterion for when an affine invariant submanifold is a branched covering construction. In Section~\ref{CD} we will discuss Alex Wright's cylinder deformation theorem~\cite{Wcyl} and related constructions. In Section~\ref{Boundary1} we discuss Maryam Mirzakhani and Alex Wright's translation surface degeneration theorem~\cite{MirWri}. In Section~\ref{Boundary2} we specialize these results to the setting of hyperelliptic components of strata. These sections establish the tools needed to run the basic mechanism of the proof: find a horizontally periodic translation surface in an affine invariant submanifold $\M$, use the cylinder deformation theorem to degenerate it to the boundary of $\M$, use the results of Section~\ref{LindseyTrees} to show that the boundary translation surface is a disjoint union of translation surfaces in hyperelliptic components of strata, and then use induction and the degeneration theorem to study the original translation surface.

In Section~\ref{ODOC} we kick off the induction argument by establishing a host of nice properties satisfied by odd dimensional orbit closures in $\hyp(g-1,g-1)$. This is leveraged in Section~\ref{S:dim4} where we show that four-dimensional affine invariant submanifolds are branched covering constructions of $\mathcal{H}(2)$. This proof is representative of the more general proof, but without the technical difficulties. In Section~\ref{EDOC} we study the flat geometry of translation surfaces in higher rank even complex-dimensional affine invariant submanifolds; this result represents the technical core of the paper. In Section~\ref{BCCC} we use the results of the preceding section to implement the strategy developed in Section~\ref{S:dim4} to the more general setting. Finally, we prove the main theorem in Section~\ref{S:bcc-criterion}.

\bigskip

\ackn The author thanks Alex~Eskin, Alex~Wright, and David~Aulicino for their insightful and extensive comments. He thanks Kathryn~Lindsey, Martin~M\"oller, Anton~Zorich, and Elise~Goujard for helpful conversations. This material is based upon work supported by the National Science Foundation Graduate Research Fellowship Program under Grant No. DGE-1144082. The author gratefully acknowledges their support.

%
%

\section{Background}\label{S:Background}

The cotangent bundle of the moduli space $\M_g$ of smooth genus $g$ curves is naturally identified with the space of quadratic differentials over $\M_g$. Each quadratic differential on a Riemann surface $X$ associates a natural flat structure - i.e. a metric that is flat away from finitely many cone points - to $X$, see Zorich~\cite{Z}. This flat structure is called a half-translation surface structure since it endows $X$ with an atlas of charts to $\C$ with transition functions given by $z \ra \pm z + c$ for some $c \in \C$. The $\GL_2 \R$ action on $\C$ induces a $\GL_2 \R$ action on half-translation surfaces and hence a $\GL_2 \R$ action on the cotangent bundle of $\M_g$. 

Teichm\"uller's theorem states that if $X$ and $Y$ are distinct Riemann surfaces in $\M_g$ that are distance $d$ apart in the Teichm\"uller metric then the geodesic from $X$ to $Y$ is given by fixing a quadratic differential $q$ on $X$, associating the natural flat structure to $(X, q)$, and then applying the matrix $\begin{pmatrix} e^d & 0 \\ 0 & e^{-d} \end{pmatrix}$. The $g_t := \begin{pmatrix} e^t & 0 \\ 0 & e^{-t} \end{pmatrix}$ action on the space of quadratic differentials is called the Teichm\"uller geodesic flow. The $\GL_2 \R$ action on the bundle of quadratic differentials is the smallest group action generated by complex scalar multiplication and Teichm\"uller geodesic flow. 

The cotangent bundle of moduli space contains a subbundle $\Omega \M_g$ of quadratic differentials that are squares of abelian differentials. $\Omega \M_g$ is stratified by specifying the number of zeros and their degree of vanishing on the underlying one-form. Let $\strata$ be such a stratum and let $(X, \omega) \in \strata$. Let $S$ be a  basis of relative homology $H_1(X; Z(\omega))$ where $Z(\omega)$ is the zero set of $\omega$. Local coordinates around $(X, \omega)$ are given by the map $\Phi(Y, \eta) = \left( \int_s \eta \right)_{s \in S}$. These coordinates are called period coordinates. 

Kontsevich and Zorich classified the connected components of strata of abelian differentials in~\cite{KZ}.
\begin{thm}[Kontsevich-Zorich]
Each stratum has at most three components, which are distinguished by hyperellipticity and spin parity.
\end{thm} 

The only two strata that admit hyperelliptic components are $\Omega \M_g(2g-2)$ and $\Omega \M_g(g-1,g-1)$. These components will be denoted by $\hyp(2g-2)$ and $\hyp(g-1,g-1)$ respectively. To be completely explicit, $\hyp(2g-2)$ coincides with $\Omega \mathcal{H}_g(2g-2)$ and $\hyp(g-1,g-1)$ is the component of $\Omega \mathcal{H}_g(g-1,g-1)$ where the two zeros are exchanged by the hyperelliptic involution. The motivation of this work is to understand the dynamics of the $\GL_2 \R$ action on the two hyperelliptic components of strata by classifying $\GL_2 \R$ orbit closures. 

The two hyperelliptic components admit another pleasantly simple description. For every genus $g$ hyperelliptic Riemann surface $X$ there is a degree $2g+1$ or $2g+2$ polynomial with complex coefficients and simple roots so that $X$ is the normalization of the projectivization of the affine curve $\{ (x, y) \in \C^2 : y^2 = f(x) \}$. Given any such affine curve, there is an associated holomorphic one-form $\frac{dx}{y}$. The abelian differentials in $\hyp(2g-2)$ (resp. $\hyp(g-1,g-1)$) are all pairs of Riemann surfaces and one-forms constructed above where $\deg f = 2g+1$ (resp. $2g+2$). Another way of phrasing the goal of this paper is to take a simple polynomial, associate to it a hyperelliptic Riemann surface and holomorphic one-form, and to study the complex geodesic that this generates.

Work of Eskin, Mirzakhani, and Mohammadi implies that $\GL_2 \R$ orbit closures in strata of abelian differentials are orbifolds:

\begin{thm}[Eskin-Mirzakhani~\cite{EM}; Eskin-Mirzakhani-Mohammadi~\cite{EMM}]
$\GL_2 \R$ orbit closures in strata of abelian differentials are affine invariant submanifolds, i.e. $\GL_2 \R$-invariant orbifolds (possibly with self-intersections) that are locally cut out by real homogeneous linear equations in period coordinates.
\end{thm} 
\noindent For a survey of this theorem and its applications to the study and classification of affine invariant submanifolds see Wright~\cite{Wsurvey}. 

The tangent bundle of an affine invariant submanifold $\M$ at a point $(X, \omega)$ is naturally identified with a complex linear subspace of $H^1(X, Z(\omega); \C)$, see~\cite{Wcyl} for details. Let $p: H^1(X, Z(\omega); \C) \ra H^1(X, \C)$ be the projection from relative to absolute cohomology.

\begin{thm}[Avila-Eskin-M\"oller~\cite{AEM}]
If $\M$ is an affine invariant submanifold, then $p\left( T_{(X, \omega)} \M \right)$ is a complex symplectic vector space. 
\end{thm}

\begin{rem}
In fact more is true, by Filip~\cite{Fi2} (Theorem 1.1) $p\left( T_{(X, \omega)} \M \right)$ respects the Hodge structure on $H^1(X, \C)$
\end{rem}

Define the rank of $\M$ to be $\mathrm{rk}(\M) := \frac{1}{2} \dim_\C p \left( T_{(X, \omega)} \M \right)$ and define $\mathrm{rel}(\M) := \dim_\C \M - 2 \cdot \mathrm{rk}(\M)$. An affine invariant submanifold is said to be higher rank if its rank is larger than one.

\subsection{Self-intersecting orbit closures and tangent spaces}

It is important to discuss an issue arising from the fact that affine invariant submanifolds may contain self-intersections. If $\M$ is an affine invariant submanifold, then the self-intersection locus is a proper affine invariant submanifold. Throughout the sequel, we will wish to refer to the tangent space of an affine invariant submanifold, but this notion may break down at the self-intersection locus or at orbifold points. To avoid orbifold issues we will assume throughout that a tacit level three structure has been fixed. The self-intersection issue is a little trickier.

One approach to this issue would be to only work with translation surfaces that are generic in $\M$ with respect to the $\GL(2, \R)$ action. However, this can become cumbersome, and so we elect to follow the approach outlined in Section 2.1 of Lanneau-Nguyen-Wright~\cite{LNW}. Let $\M$ be an affine invariant submanifold and let $\M'$ be the set of translation surfaces $(X, \omega)$ together with a maximal subspace $V \subseteq H^1(X, \Sigma; \C)$ that is tangent to $\M$ and where $\Sigma$ is the collection of zeros of $\omega$. Notice that $\M'$ has a $\GL_2(\R)$ action that acts in the usual way on $(X, \omega)$ and by parallel transport using the Gauss-Manin connection on $V$. Let $f: \M' \ra \M$ be the map that forgets the subspace $V$.  This map is generically one-to-one and $\GL_2(\R)$-equivariant. 

Given an element $v$  sufficiently close to $0$ in relative cohomology $H^1(X, \Sigma; \C)$ let $(X, \omega) + v$ be the unique element of the stratum near $(X, \omega)$ where $\omega$ is replaced by $\omega + v$. The condition that $v$ is sufficiently close to $0$ is necessary to prevent the translation surface from degenerating. A neighborhood basis of $\M'$ around $(X, \omega; V)$ will be the following. For any open set $U \subseteq V$ containing $0$ and where $\{ (X, \omega) + u \}_{u \in U}$ is a contractible set contained in $\M$ set the corresponding neighborhood of $(X, \omega; V)$ to be $\{ (X, \omega + u; V) : u \in U \}$. With this topology, $f$ is continuous and $\M'$ has a linear structure that makes $f$ locally linear. Finally, the tangent space at a point $(X, \omega; V)$ of $\M'$ is canonically identified with $V$ and so $\M'$ is smooth. 

Throughout the text we will tacitly work on $\M'$, but write $\M$. This allows us to refer to a well-defined tangent space of an affine invariant submanifold (despite the fact that $\M$ may actually have self-intersection in a stratum). For a discussion of how to avoid issues of self-intersection when using the Mirzakhani-Wright partial compactification see the statement of Theorem 2.9 in Mirzakhani-Wright~\cite{MirWri} and the surrounding discussion.

\section{The Lindsey Half-Tree of Horizontally Periodic Translation Surfaces in Hyperelliptic Components of Strata}\label{LindseyTrees} 

In this section we will review a construction of Lindsey that associates a half-tree to a horizontally periodic translation surface in a hyperelliptic component of a stratum of abelian differentials. We will show that any surface constructed in this way is guaranteed to be in $\hyp(2g-2)$ or $\hyp(g-1,g-1)$ for some $g \geq 1$. In particular, translation surfaces constructed in this way will have marked points if and only if the genus is one. 

We begin by making two combinatorial definitions. A half-graph $\Gamma$ consists of a set of vertices, a set of edges (each of which connects two vertices), and a set of half-edges that begin at a vertex but do not end at a vertex. A half-tree is a half-graph whose vertex and edge set form a tree. 

Let $(X, \omega)$ be a horizontally periodic translation surface in a hyperelliptic component of a stratum of abelian differentials. Lindsey in~\cite{Lindsey} showed the following results:

\begin{thm}[Lindsey~\cite{Lindsey} Lemma 2.2]
 Every horizontal cylinder on $(X, \omega)$ is fixed by the hyperelliptic involution. Therefore, if $C$ and $D$ are two horizontal cylinders and $C$ shares a saddle connection with $D$ on its top boundary, then it shares a saddle connection with $D$ on its bottom boundary as well.
 \end{thm}
 
 \begin{thm}[Lindsey~\cite{Lindsey} Lemma 2.4]
To each cylinder $C$ one may associate a translation surface in the following way. Consider $C$ as a cylinder in $(X, \omega)$ with boundary. If the hyperelliptic involution exchanges two saddle connections on the top and bottom boundary of $C$, then glue them together by translation. The resulting translation surface is hyperelliptic. 
\end{thm}

\noindent These two observations lead to the following construction of a half-tree associated to $(X, \omega)$. To each horizontal cylinder associate a vertex. Connect two distinct vertices by an edge if the corresponding cylinders share a saddle connection on their boundary. By Lindsey~\cite{Lindsey} Lemma 2.4, the resulting graph is a tree. To make the graph into a half-tree and not just a tree to each cylinder that has a saddle connection joining its top and bottom boundary, add a half-edge to the corresponding vertex of $\Gamma$. The total number of half-edges is $2g+|\Sigma| - 2$ where $\Sigma$ is the singular set and where we count each full edge as two half-edges. We will call $\Gamma$ the Lindsey tree associated to $(X, \omega)$.

Lindsey's result are actually more general than the results we have stated here. It associates a tree to any translation surface in a hyperelliptic component; in particular the translation surface need not be horizontally periodic. In this more complicated construction, each node represents either a horizontal cylinder or a minimal component of the horizontal line flow and a new kind of half-edge is required corresponding to horizontal lines beginning at a singularity, but never terminating at a singularity.

In Figures~\ref{fig:LT1} and~\ref{fig:LT2} below are all possible half-trees with four or fewer half-edges - i.e. the ones arising from surfaces in $\strata(2)$ and $\strata(1,1)$ - and corresponding horizontally periodic surfaces.

\begin{figure}[h]
\centering
    \begin{tikzpicture}
		\draw (0,0) -- (0,1) -- (3,1) -- (3,0) -- (0,0);
		\draw[dotted] (1,1) -- (1,0);
		\draw[dotted] (2,1) -- (2,0);
		\node at (.5,1.4) {$1$}; \node at (.5,-0.4) {$3$};
		\node at (1.5,1.4) {$2$}; \node at (1.5,-0.4) {$2$};
		\node at (2.5,1.4) {$3$}; \node at (2.5,-0.4) {$1$};
		
		\draw (0, -2) -- (1,-2) -- (2, -1) -- (1, -2) -- (2, -3);
		\draw[black, fill] (1,-2) circle[radius=4.5pt];
	\end{tikzpicture}
	\hspace{5mm}
        \begin{tikzpicture}
		\draw (0,0) -- (0,2) -- (1,2) -- (1,1) -- (2,1) -- (2,0) -- (0,0);
		\draw[dotted] (0,1) -- (1,1) -- (1,0);
		\node at (.5, 2.4) {$1$}; \node at (.5,-0.4) {$2$};
		\node at (1.5,1.4) {$2$}; \node at (1.5,-0.4) {$1$};
		
		\draw (0,-2) -- (1,-2) -- (2,-2);
		\draw[black, fill] (0,-2) circle[radius=4.5pt];
		\draw[black, fill] (1,-2) circle[radius=4.5pt];
	\end{tikzpicture}
	\caption{Lindsey trees of horizontally periodic translation surfaces in $\strata(2)$}
	\label{fig:LT1}
\end{figure}
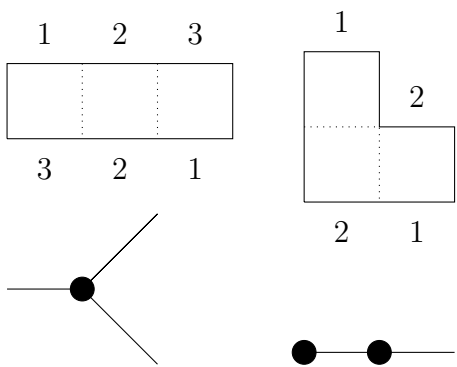

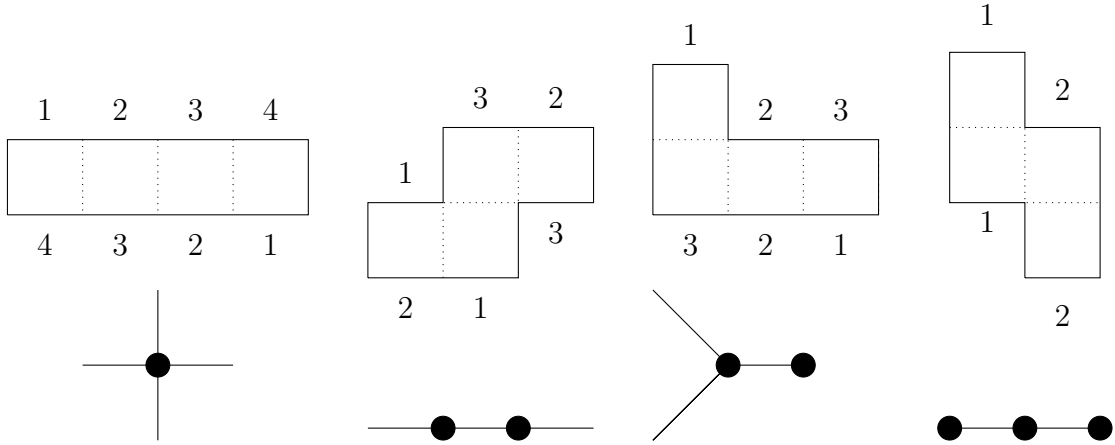
\begin{figure}
	\begin{tikzpicture}
        
		\draw (0,0) -- (0,1) -- (4,1) -- (4,0) -- (0,0);
		\draw[dotted] (1,1) -- (1,0);
		\draw[dotted] (2,1) -- (2,0);
		\draw[dotted] (3,1) -- (3,0);
		\node at (.5,1.4) {$1$}; \node at (.5,-0.4) {$4$};
		\node at (1.5,1.4) {$2$}; \node at (1.5,-0.4) {$3$};
		\node at (2.5,1.4) {$3$}; \node at (2.5,-0.4) {$2$};
		\node at (3.5,1.4) {$4$}; \node at (3.5,-0.4) {$1$};

		\draw (2, -1) -- (2,-3);
		\draw (1,-2) -- (3, -2);  
		\draw[black, fill] (2,-2) circle[radius=4.5pt];
	\end{tikzpicture}
	\hspace{5mm}
	\begin{tikzpicture}
        
		\draw (0,0) -- (0,1) -- (1,1) -- (1,2) -- (3,2) -- (3, 1) -- (2,1) -- (2,0) -- (0,0); 
		\draw[dotted] (1,0) -- (1,1) -- (2,1) -- (2,2);
		\node at (.5, 1.4) {$1$}; \node at (.5,-0.4) {$2$};
		\node at (1.5,2.4) {$3$}; \node at (1.5,-0.4) {$1$};
		\node at (2.5,2.4) {$2$}; \node at (2.5,.6) {$3$};
		
		\draw (0,-2) -- (1, -2) -- (2,-2) -- (3,-2);
		\draw[black, fill] (1,-2) circle[radius=4.5pt];
		\draw[black, fill] (2,-2) circle[radius=4.5pt];

	\end{tikzpicture}
	\hspace{5mm}
	\begin{tikzpicture}
        
		\draw (0,0) -- (0,2) -- (1,2) -- (1,1) -- (3,1) -- (3, 0) -- (0,0); 
		\draw[dotted] (0,1) -- (1,1) -- (1,0);
		\draw[dotted] (2,1) -- (2,0);
		\draw[dotted] (3,1) -- (3,0);
		\node at (.5, 2.4) {$1$}; \node at (.5,-0.4) {$3$};
		\node at (1.5,1.4) {$2$}; \node at (1.5,-0.4) {$2$};
		\node at (2.5,1.4) {$3$}; \node at (2.5,-0.4) {$1$};

		\draw (0,-1) -- (1, -2) -- (0,-3) -- (1,-2) -- (2,-2);
		\draw[black, fill] (1,-2) circle[radius=4.5pt];
		\draw[black, fill] (2,-2) circle[radius=4.5pt];

	\end{tikzpicture}
	\hspace{5mm}
	\begin{tikzpicture}
        
        		\draw (1,0) -- (1,1) -- (0,1) -- (0,3) -- (1,3) -- (1,2) -- (2,2) -- (2,0) -- (1,0);
		\draw[dotted] (0,2) -- (1,2) -- (1,1) -- (2,1);
		\node at (.5, 3.5) {$1$}; \node at (1.5, 2.5) {$2$};
		\node at (.5, .75) {$1$}; \node at (1.5, -.5) {$2$};

		\draw (0,-2) -- (1, -2) -- (2,-2);
		\draw[black, fill] (0,-2) circle[radius=4.5pt];
		\draw[black, fill] (1,-2) circle[radius=4.5pt];
		\draw[black, fill] (2,-2) circle[radius=4.5pt];
	\end{tikzpicture}
        
\caption{Lindsey trees of horizontally periodic translation surfaces in $\strata(1,1)$}
\label{fig:LT2}
\end{figure}

Define the combinatorial type of a horizontally periodic translation surface in $\mathcal{H}^{hyp}(2g-2)$ or $\mathcal{H}^{hyp}(g-1,g-1)$ to be the equivalence class of horizontally periodic translation surfaces that are related by some combination of the following three operations:
\begin{enumerate}
\item Horizontally shearing a horizontal cylinder.
\item Vertically dilating a horizontal cylinder.
\item Changing the lengths of a collection of horizontal saddle connections that is invariant under the hyperelliptic involution.
\end{enumerate}
The notion of combinatorial type applies equally well to horizontal cylinders on translation surfaces.

\begin{sslem}\label{easy}
Each node in the Lindsey half-tree corresponds to a horizontal cylinder of the same combinatorial type as the cylinder in Figure~\ref{fig:CT}.
\begin{figure}[H]
\centering
\begin{tikzpicture}[scale=.75]
	\draw (0,0) -- (4,0) -- (4,1) -- (0,1) -- (0,0);
	\draw[dotted] (1,0) -- (1,1);
 	\draw[dotted] (2,0) -- (2,1);
	\draw[dotted] (3,0) -- (3,1);
	\node at (.5, -.5) {$1$}; \node at (.5,1.5) {$n$};
	\node at (1.5, -.5) {$2$}; \node at (1.5,1.5) {$n-1$};
	\node at (2.5, -.5) {$\hdots$}; \node at (2.5, 1.5) {$\hdots$};
	\node at (3.5, -.5) {$n$}; \node at (3.5, 1.5) {$1$};
\end{tikzpicture}
\caption{The combinatorial type of a horizontal cylinder in $\hyp(2g-2)$ or $\hyp(g-1,g-1)$}
\label{fig:CT}
\end{figure}
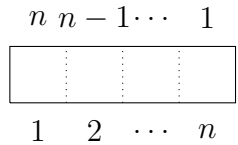
By identifying horizontal saddle connections that have the same label, we may choose to view this cylinder as a translation surface in a hyperelliptic component.
\end{sslem}
\begin{proof}
By Lindsey (\cite{Lindsey} Lemma 2.4) each node corresponds to a hyperelliptic surface that is contained in one horizontal cylinder and whose hyperelliptic involution is given by rotation by $\pi$. It follows that each node has the combinatorial type shown above.
 \end{proof}

So that we may refer to it later the combinatorial type of the cylinder in Lemma~\ref{easy} will be referred to as hyperelliptic combinatorial type.

\begin{ssthm}\label{LT}
The combinatorial types of horizontally periodic translation surfaces in $\mathcal{H}^{hyp}(2g-2)$ and $\mathcal{H}^{hyp}(g-1,g-1)$ are in bijective correspondence with planar embeddings of half-trees with $2g+|\Sigma| - 2$ half-edges up to precomposition with half-tree automorphisms (here we count full edges as two half-edges).
\end{ssthm}
\begin{proof}
For the forward direction of this correspondence, take the combinatorial type of a horizontally periodic translation surface in  $\mathcal{H}^{hyp}(2g-2)$ or $\mathcal{H}^{hyp}(g-1,g-1)$ and associate to it the Lindsey half-tree. We form the planar embedding of the Lindsey half-tree (up to automorphism) by cyclically ordering the half-edges attached to a node according to the cyclic ordering of saddle connections on the upper boundary of the horizontal cylinder corresponding to the node.

For the reverse direction we must take a half-tree $\Gamma$ and produce a translation surface $(X_\Gamma, \omega_\Gamma)$ in a hyperelliptic component of a stratum whose associated Lindsey half-tree is $\Gamma$. Associate to each node $v \in \Gamma$ the translation surface in Figure~\ref{fig:CT} taking $n$ to be the number of edges and half-edges attached to $v$. Label the edges and half-edges attached to $v$ clockwise $\{1,\hdots, n\}$. If an edge of the Lindsey tree connects the nodes $v$ and $w$ it will have two labels $i$ and $j$ coming from $v$ and $w$ respectively. To form $(X_\Gamma, \omega_\Gamma)$ slice open $v$ along saddle connection $i$ and $w$ along saddle connection $j$ and glue $i$ to $j$. Do this for all edges in $\Gamma$. 

The resulting surface has an involution given by $-I$ that fixes every horizontal cylinder. Since the surface in Figure~\ref{fig:CT} is hyperelliptic with hyperelliptic involution given by $-I$ the quotient of every node of the Lindsey tree is a copy of $\p^1$. The tree structure (ignoring half-edges) of the half-tree describes how the copies of $\p^1$ glue together. Since trees are contractible it follows that $(X_\Gamma, \omega_\Gamma)/-I$ is homeomorphic to $\p^1$. Consequently, $(X_\Gamma, \omega_\Gamma)$ is a hyperelliptic translation surface, i.e. $X_\Gamma$ admits a hyperelliptic involution that takes $\omega_\Gamma$ to $-\omega_\Gamma$.  To conclude that $(X_\Gamma, \omega_\Gamma)$ actually lies in  $\mathcal{H}^{hyp}(2g-2)$ or $\mathcal{H}^{hyp}(g-1,g-1)$ and not in a hyperelliptic locus in another connected component it suffices to show that the surface has either one zero or two zeros that are exchanged by the hyperelliptic involution.

The translation surfaces in Figure~\ref{fig:CT} has a single Weierstrass point at the center of the rectangle and another $n+1$ Weierstrass points at the midpoints of each saddle connection on the boundary (whether or not the vertex is also a Weierstrass point will depend on the parity of $n$).  Therefore, prior to identifying saddle connections, each node $v$ of $\Gamma$ contributes $d_v + 2$ Weierstrass points (not including potential Weierstrass points at vertices) where $d_v$ is the number of full and half-edges attached to $v$. When two saddle connections on nodes $v$ and $w$ are identified the midpoints of both saddle connections are exchanged by the hyperelliptic involution and hence cease to be fixed points. It follows that the number of Weierstrass points on $(X_\Gamma, \omega_\Gamma)$ is
\[ 2g+2 = \left( \sum_{v \in V} d_v + 2 \right) - 2|E| + \# \{ \text{zeros fixed by the hyperelliptic involution} \} \]
where $V$ (resp. $E$) are the vertices (resp. full edges) of $\Gamma$.  Since $\Gamma$ is a tree $2|V| - 2|E| = 2$ and since the total number of half-edges (counting each full edge as two half edges) is $2g+|\Sigma|-2$ we also have that $\sum_{v \in V} d_v = 2g+|\Sigma|-2$. Therefore,
\[ 2g+2 = \left( 2g + |\Sigma| - 2 +2|V| \right) -2|E| + \# \{ \text{zeros fixed by the hyperelliptic involution} \} \]
which shows that
\[ |\Sigma| + \# \{ \text{zeros fixed by the hyperelliptic involution} \} = 2 \]
It follows that $(X_\Gamma, \omega_\Gamma)$ has either one fixed zero or two zeros that are exchanged by the hyperelliptic involution; so $(X_\Gamma, \omega_\Gamma)$ lies in $\mathcal{H}^{hyp}(2g-2)$ or $\mathcal{H}^{hyp}(g-1,g-1)$.
 \end{proof}
 
\begin{sscoro}\label{LTC}
Suppose that $(X, \omega)$ is a horizontally periodic translation surface satisfying the following two conditions:
\begin{enumerate} 
\item Every horizontal cylinder has hyperelliptic combinatorial type 
\item The cylinder diagram of $(X, \omega)$ in the horizontal direction is a Lindsey tree 
\end{enumerate}
then $(X, \omega)$ lies in a hyperelliptic component of a stratum of abelian differentials.
\end{sscoro}
\begin{proof}
Suppose that $(X, \omega)$ was constructed as described. Let $\Gamma$ be the cylinder diagram. Shearing individual cylinders preserves the stratum in which $(X, \omega)$ is contained. Moreover, given a horizontal cylinder and a saddle connection $s$ on its top boundary there is a saddle connection $s'$ on its lower boundary that has identical length by construction. Changing the length of $s$ and $s'$ so that they remain of identical length preserves the stratum in which $(X, \omega)$ is contained. These two operations can be successively used to move $(X, \omega)$ to the surface $(X_\Gamma, \omega_\Gamma)$, which we proved belonged to a hyperelliptic component in Theorem~\ref{LT}.
\end{proof}

%
%

\section{Branched Covering Constructions}\label{BCCs}
A translation covering $f: (Y, \eta) \ra (X, \omega)$ is a holomorphic map $f: Y \ra X$ such that $f^*\omega = \eta$. A simple translation covering is a translation covering that is branched over the zeros of $\omega$ and for which $Y$ is connected. The goal of this section is to develop a criterion to recognize ``branched covering constructions over $\M$" - i.e. affine invariant submanifolds all of whose elements are simple translation coverings of elements in a component $\M$ of a stratum of abelian differentials. We begin extending Mumford's compactness theorem to strata of abelian differentials. 

\begin{sslem}[Maskit-Mumford Compactness Lemma]\label{compactness}
If $\left( (X_n, \omega_n) \right)_n$ is a sequence of translation surfaces in a fixed stratum that have area bounded from above and below and the length of their shortest saddle connection bounded below, then there is a convergent subsequence.
\end{sslem}
\begin{proof}
By Maskit (\cite{Maskit} Corollary 2) if the length of the shortest saddle connection is bounded away from zero then the length of the shortest hyperbolic curve on $X_n$ is bounded away from zero. By Mumford's compactness theorem there is a convergent subsequence of $X_n$. Passing to this subsequence let $X$ be the limit and let $U$ be a precompact neighborhood of $X$ on which the bundle of holomorphic one-forms is trivial. Since the area is bounded below and above $(X_n, \omega_n)$ eventually is contained in a bundle of compact annuli over the compact set $\overline{U}$. Therefore there is a convergent subsequence. Since no saddle connection becomes short the sequence remains in the same fixed stratum.
 \end{proof}
 
 \begin{ssthm}\label{useful0}
Suppose that $f: (X, \omega) \ra (Y, \eta)$ is a simple translation covering. Let $\M$ be the $\GL_2 \R$ orbit closure of $(X, \omega)$ and let $\N$ be the $\GL_2 \R$ orbit closure of $(Y, \eta)$. Every translation surface in $\M$ is a translation covering of a surface in $\N$. 
\end{ssthm}
 \begin{proof}
Let $\M$ be the orbit closure of $(X, \omega)$ and $\N$ the orbit closure of $(Y, \eta)$. First notice that if $f: (X, \omega) \ra (Y, \eta)$ is a simple translation covering, then for any $g$ in $\GL_2(\R)$ we have that $gfg^{-1}: g(X, \omega) \ra g(Y, \eta)$ is a simple translation covering too. Therefore, to show that the orbit closure of $(X, \omega)$ only includes surfaces that are simple translation coverings of a surface in $\M$ it suffices to show that if $(g_i)$ is a sequence of elements of $\GL_2(\R)$ and $g_i(X, \omega)$ converges to $(X', \omega')$ then $(X', \omega')$ is a simple translation cover of some translation surface in $\N$. Let $(X_i, \omega_i) = g_i(X, \omega)$ and $(Y_i, \eta_i) = g_i (Y, \eta)$. Let $f_i = g_i f g_i^{-1}$. 

Notice that since each $f_i$ has the same degree, say $d$, the systole of $(Y_i, \eta_i)$ in the flat metric is bounded below by $\frac{\mathrm{sys}(X_i)}{d}$. Since $(X_i, \omega_i)$ converge to $(X', \omega')$ the length of the systole along the sequence $(X_i, \omega_i)$ is bounded below and hence the length of the systole along the sequence $(Y_i, \eta_i)$ is also bounded below. Since a degree $d$ map preserves area up to a factor of $\frac{1}{d}$, it follows from the Maskit-Mumford compactness lemma (Lemma~\ref{compactness}) that $(Y_i, \eta_i)$ has a convergent subsequence. After passing to this subsequence we may suppose that $(Y_i, \eta_i)$ converges to a translation surface $(Y', \eta')$ belonging to $\N$.

After deleting sufficiently many initial terms we may suppose that all $(X_i, \omega_i)$ and $(Y_i, \eta_i)$ belong to a small neighborhood of $(X', \omega')$ and $(Y', \eta')$ respectively where the zeros of the one-forms are labelled. Let $\phi_i: X' \ra X_i$ and $\psi_i: Y' \ra Y_i$ be quasiconformal maps of minimal dilatation that take labelled zeros to the corresponding labelled zeros. Suppose too, after perhaps again passing to a subsequence, that for all $i$ the ramification type over a given labelled zero of $(Y_i, \eta_i)$ is constant. Since $(X_i, \omega_i)$ and $(Y_i, \eta_i)$ converge we have that the dilatation of these $\phi_i$ and $\psi_i$ tends to $1$ as $i$ tends to infinity. Therefore, the dilatation of the map $F_i = \psi_i^{-1} \circ f_i \circ \phi_i: X' \ra Y'$ tends to $1$ as $i$ tends to infinity as well. Since a collection of quasiconformal maps of bounded dilatation is precompact, it follows that there is a subsequence of $F_i$ that tends to a quasiconformal map $F: X' \ra Y'$ of dilatation $1$, i.e. a holomorphic map. Moreover, the condition on labelled zeros implies that $\mathrm{div}(\omega') = \mathrm{div}(F^* \eta)$. Therefore, up to multiplication by scalars $F^* \eta = \omega'$. Therefore, $(X', \omega')$ is a simple translation covering of a surface in $\N$ as desired.
\end{proof}
 
 \begin{sscoro}\label{useful}
An affine invariant submanifold $\M$ is a branched covering construction if there is an $\M$-generic point that is a simple translation covering of a lower genus translation surface.
\end{sscoro}
\begin{proof}
Suppose that $(X, \omega)$ is a translation surface that is generic in $\M$ with respect to the $\GL(2, \R)$ action. Suppose that there is a map $f: (X, \omega) \ra (Y, \eta)$ that is a simple translation covering, where $(Y, \eta)$ is a lower genus translation surface. Let $\N$ be the orbit closure of $(Y, \eta)$. By Theorem~\ref{useful0} every point in $\M$ is a simple translation covering of a point in $\N$ and so $\M$ is a branched covering construction.
\end{proof}

\section{Cylinder Deformations}\label{CD} 
Throughout this section $\M$ will be an affine invariant submanifold in a component $\strata$ of a stratum. Suppose that $(X, \omega)$ is a translation surface in an affine invariant submanifold $\M$. Let $C$ and $C'$ be two cylinders on $(X, \omega)$ with core curves $\gamma_C$ and $\gamma_{C'}$. If $\gamma_C$ and $\gamma_{C'}$ are parallel in some neighborhood $U \subseteq \M$ of $(X, \omega)$, then $C$ and $C'$ are said to be $\M$-equivalent. When the affine invariant submanifold $\M$ is clear from context, $\M$-equivalent and $\M$-equivalence class will be shorted to ``equivalent" and ``equivalence class" respectively.

\begin{ssthm}[Cylinder Proportion Theorem, Proposition 3.2, Nguyen-Wright~\cite{NW}]\label{CPT}
If $C$ and $C'$ are two $\M$-equivalent cylinders and $\V$ is any equivalence class of cylinders, then 
\[ \frac{|C \cap \V|}{|C|} = \frac{|C' \cap \V|}{|C'|} \]
where $|\cdot|$ denotes area.
\end{ssthm}

Applying the matrix $u_t := \begin{pmatrix} 1 & t \\ 0 & 1 \end{pmatrix}$ to a horizontal $\M$-equivalence class of horizontal cylinders $\Cyl$ will be called (horizontally) shearing $\Cyl$. Applying the matrix $a_t:= \begin{pmatrix} 1 & 0 \\ 0 & e^t \end{pmatrix}$ will be called (vertically) dilating $\Cyl$. 

\begin{ssthm}[Cylinder Deformation Theorem, Wright~\cite{Wcyl}]\label{CDT}
Let $(X, \omega) \in \M$ be a translation surface and let $\Cyl$ be an $\M$-equivalence class of horizontal cylinders on $(X, \omega)$.  Horizontally shearing and vertically dilating $\Cyl$ by $t$ remains in $\M$ for all $t$. 
\end{ssthm}

A special feature of the hyperelliptic component of a stratum is that if two horizontal cylinders share a horizontal saddle connection then they share exactly two - one on the top of each cylinder and one on the bottom of each cylinder. This feature holds because the graph of cylinder adjacencies is a tree and because if a cylinder $C$ borders a cylinder $D$ on its top boundary, then it borders $D$ on its bottom boundary as well. The two saddle connections joining the cylinders are exchanged by the hyperelliptic involution. Given two adjacent cylinders $C$ and $D$, which border each other along saddle connections $s_1$ and $s_2$, we say that the cylinders are in transverse standard position if there is a cylinder $V$ that is contained in $C \cup D$, that contains $s_1$ and $s_2$, and that only intersects the core curves of cylinder $C$ and $D$ once. We say that $C$ and $D$ are in standard position if the core curve of $V$ is perpendicular to the core curves of cylinders $C$ and $D$. The definition is illustrated in Figure~\ref{F:tsp}.

\begin{figure}[H]
    \begin{subfigure}[b]{0.3\textwidth}
        \centering
        \resizebox{\linewidth}{!}{\begin{tikzpicture}
        		\draw (0,0) -- (2,0);
		\draw (0,1) -- (1,1);
		\draw (1,2) -- (3,2);
		\draw (2,1) -- (3,1);
		\draw[dotted] (1,0) -- (1,2);
		\draw[dotted] (2,0) -- (2,2);
		\draw[dotted] (1,1) -- (2,1);
		\node at (1.5, 2.25) {$s_1$}; \node at (1.5, -.25) {$s_1$};
		\node at (2.5, 1.5) {$C$}; \node at (.5, .5) {$D$}; \node at (1.5, 1) {$V$};
	\end{tikzpicture}
        }
        \caption{ Standard position}
    \end{subfigure}
    \qquad
        \begin{subfigure}[b]{0.3\textwidth}
        \centering
        \resizebox{\linewidth}{!}{\begin{tikzpicture}
        		\draw (0,0) -- (2,0);
		\draw (0,1) -- (1,1);
		\draw (.7,2) -- (3,2);
		\draw (2,1) -- (3,1);
		\draw[dotted] (1,0) -- (1,1) -- (.6, 2);
		\draw[dotted] (2,0) -- (2,1) -- (1.6, 2);
		\draw[dotted] (1,1) -- (2,1);
		\node at (1.2, 2.25) {$s_1$}; \node at (1.5, -.25) {$s_1$};
		\node at (2.5, 1.5) {$C$}; \node at (.5, .5) {$D$};
	\end{tikzpicture}
        }
        \caption{ Not transverse standard position  }
    \end{subfigure}
\caption{An illustration of the definition of transverse standard position} 
\label{F:tsp}
\end{figure}
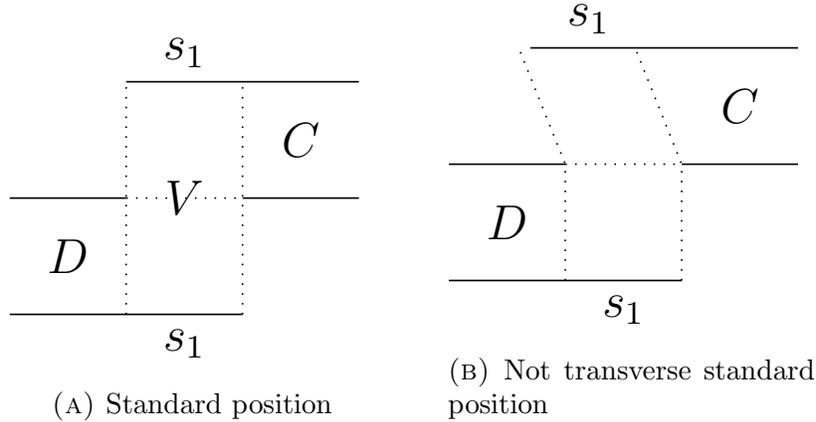

\begin{sslem}[Standard Position]\label{stp}
Suppose that $(X, \omega)$ is a translation surface in a hyperelliptic component of a stratum and suppose that $C$ and $D$ are adjacent horizontal cylinders that belong to distinct equivalence classes $\Cyl_1$ and $\Cyl_2$ respectively. 
\begin{enumerate} 
\item It is possible to shear $\Cyl_2$ so that $C$ and $D$ are in transverse standard position.
\item It is possible to shear $\Cyl_1$ and $\Cyl_2$ so that $C$ and $D$ are in standard position.
\end{enumerate}
\end{sslem}
\begin{proof}
Let $s_1$ and $s_2$ be the horizontal saddle connections lying on the boundary of $C$ and $D$. By Theorem~\ref{CDT} shear $\Cyl_1$ so that $s_1$ lies directly over $s_2$; then shear $\Cyl_2$ so that $s_2$ lies directly over $s_1$. Recall that $s_1$ and $s_2$ are exchanged by the hyperelliptic involution and hence have identical lengths. Choose $V$ to be the vertical cylinder passing through $s_1$ and $s_2$. This proves part (2), the proof of part (1) is almost completely identical.
 \end{proof}
 
In the sequel, moving to the second configuration will be called putting $C$ and $D$ in standard position. Moving to the first will be called putting $C$ and $D$ in transverse standard position while fixing $C$.

If $C$ is a cylinder call the distance $h_C$ from one boundary of the cylinder to the other its height. Let $\gamma_C^*$ be the cohomology class that is dual to the core curve of $C$ under the intersection pairing. This cohomology class requires specifying an orientation on $\gamma_C$. Usually this orientation will not matter, but we will establish the conventions that when $C$ is horizontal the orientation is left to right, when $C$ is vertical it is top to bottom, and when $C_1, \hdots, C_n$ are all $\M$-equivalent cylinders we will assume that the holonomy vectors of the core curves point in the same directions.

Let $\Cyl_1, \hdots, \Cyl_n$ be an enumeration of the $\M$-equivalence classes of horizontal cylinders on $(X, \omega)$. For each equivalence class $\Cyl_i$ there is an element of the tangent space called the standard shear which is defined to be $u_{\Cyl_i} = \sum_{c \in \Cyl_i} h_c \gamma_c^*$. A reformulation of the cylinder deformation theorem is that the standard shear is always in the tangent space of $\M$ at $(X, \omega)$. 

Let $\Cyl$ denote the collection of all horizontal cylinders on $(X, \omega)$. The twist space of $\M$ at $(X, \omega)$ is defined to be
\[ \Twist \M := \mathrm{span}_\R \left( \gamma_c^* \right)_{c \in \Cyl} \cap T^\R_{(X, \omega)} \M \]
where $T^\R_{(X, \omega)} \M = T_{(X, \omega)} \M \cap H^1(X, Z(\omega); \R)$ where $T_{(X, \omega)} \M$ has been identified with a subspace of $H^1(X, Z(\omega); \C)$. The standard shears are always in the twist space. Define the cylinder preserving space, denoted $\Pres \M$, to be all elements of $T^\R_{(X, \omega)} \M$ that pair trivially with every element of $(\gamma_c)_{c \in \Cyl}$ under the intersection pairing. It is clear that $\Twist \M \subseteq \Pres \M$. The following theorem establishes that there is always a translation surface $(X, \omega)$ in an affine invariant submanifold $\M$ where $\Twist \M = \Pres \M$.

\begin{ssthm}[Lemma 8.6, Wright~\cite{Wcyl}]\label{TP}
$\Twist \M = \Pres \M$ whenever $(X, \omega)$ has the maximum number of horizontal cylinders in $\M$.
\end{ssthm}

The next theorem indicates why having $\Twist \M = \Pres \M$ is special. In particular, it says that whenever equality is achieved $(X, \omega)$ has at least $\mathrm{rk}(\M)$ many $\M$-equivalence classes and the twist space projects to a Lagrangian in $p\left(T_{(X, \omega)} \M \right)$.

\begin{ssthm}[Lemma 8.12, Wright~\cite{Wcyl}]
If $(X, \omega)$ is a translation surface in $\M$ and $\Twist \M = \Pres \M$ then $\mathrm{span}_\R p\left( u_{\Cyl_i} \right)_{i=1}^n$ is a Lagrangian subspace of $p(T_{(X, \omega)}^\R \M)$ where $\{ \Cyl_1, \hdots, \Cyl_n \}$ is an enumeration of the $\M$-equivalence classes of horizontal cylinders and $u_{\Cyl_i}$ is the standard shear. In particular, $(X, \omega)$ contains at least $\mathrm{rk}(\M)$ distinct $\M$-equivalence classes of horizontal cylinders.
\end{ssthm}

The combination of the previous two theorems is an engine that allows us to convert the rank of an affine invariant submanifold into geometric information that picks out a translation surface where a large dimensional subspace, the twist space, of the tangent space is geometrically meaningful. Recall that, given a translation surface $(X, \omega)$ belonging to a stratum $\mathcal{H}$ and with cone points $\Sigma$, the tangent space to $\strata$ at $(X, \omega)$ can be identified with the relative cohomology group $H^1(X, \Sigma; \C)$. Let $p$ be the projection from $H^1(X, \Sigma; \C)$ onto absolute cohomology. 

\begin{ssthm}[Twist Space Decomposition Theorem, Theorem 4.7, Mirzakhani-Wright~\cite{MirWri}]\label{keep}
Let $(X, \omega)$ be a translation surface in an affine invariant submanifold $\M$. Let $\Cyl_1, \hdots, \Cyl_d$ be an enumeration of the $\M$-equivalence classes of horizontal cylinders. 
\begin{enumerate}
\item If $v \in \Twist \M$ then $v = \sum_{i=1}^d v_i$ where $v_i \in \Twist \M \cap \mathrm{span}_\R (\gamma_c^*)_{c \in \Cyl_i}$.
\item If $v_i \in  \Twist \M \cap \mathrm{span}_\R (\gamma_c^*)_{c \in \Cyl_i}$ then $v_i \in \R \cdot u_{\Cyl_i} \oplus \ker p$ where $u_{\Cyl_i}$ is the standard shear.
\end{enumerate}
\end{ssthm}

The last result regarding cylinder deformations that we need is the statement that given a collection of $d$ $\M$-equivalence classes of cylinders, where $d$ is no bigger than the rank of $\M$, it is possible to perturb the translation surface so that one $\M$-equivalence class becomes disjoint and vertical and all others remain horizontal. This result is crucial in establishing that all higher rank affine invariant submanifolds of complex dimension four in hyperelliptic components of strata are branched covering constructions over $\strata(2)$.

\begin{ssthm}[Perturbation Theorem, Lemma 5.5, Mirzakhani-Wright~\cite{MirWri}]\label{perturb}
Suppose that $(X, \omega)$ is a translation surface in an affine invariant submanifold $\M$. Suppose that $\Cyl_1, \hdots, \Cyl_d$ are $\M$-equivalence classes of horizontal cylinders such that $\{ p(u_{\Cyl_1}), \hdots, p(u_{\Cyl_d}) \}$ spans a $d$ dimensional subspace. Define $\Cyl:= \Cyl_1 \cup \hdots \cup \Cyl_{d-1}$. There is a piecewise smooth path $f:[0, 1] \ra \M$ such that $f(0) = (X, \omega)$ and along the path
\begin{enumerate}
\item All cylinders in $\Cyl$ persist and are horizontal.
\item The cylinders in $\Cyl_d$ persist, become nonhorizontal on $f(t)$ for $t>0$, and vertical on $f(1)$.
\item At all points along the path any cylinder $\M$-equivalent to $\Cyl_d$ is disjoint from any cylinder $\M$-equivalent to $\Cyl_i$ for all $i \in \{1, \hdots, d-1\}$.
\end{enumerate}
\end{ssthm}
\begin{proof}
Let $\gamma_i$ be the core curve of some cylinder in $\Cyl_i$ for each $i \in \{1, \hdots, d\}$. Consider the linear functionals $\left( \gamma_i \right)_{i=1}^{d-1}$ on $T_{(X, \omega)} \M$. Since the linear functionals factor through $p: T_{(X, \omega)} \M \ra H^1(X; \C)$ we see that the intersection of the kernel of these functionals on $p(T_{(X, \omega)} \M)$ is at least dimension $(2r-d)+1$. Therefore there is a vector $v \in T_{(X, \omega)} \M$ that is not in the kernel of $p$, not in the cylinder preserving space, and such that $v(\gamma_i) = 0$ for $1 \leq i \leq d-1$. 

Consider the path $(X, \omega) + tv$ for $t \geq 0$. This path is well-defined and remains in $\M$ for some range $t \in [0, T]$. Since $v(\gamma_i)=0$ for each $1 \leq i \leq d-1$ it follows by definition of $\M$-equivalence that $\Cyl_i$ persist (perhaps after decreasing $T$) and remain horizontal for $1 \leq i \leq d-1$. Since $v$ is not in the cylinder preserving space we see that, perhaps after decreasing $T$, $\Cyl_d$ also persists and becomes non-horizontal. By perhaps decreasing $T$ again we may suppose by Mirzakhani-Wright Lemma 5.1~\cite{MW} that any cylinder $\M$-equivalent to $\Cyl_d$ remains disjoint from any cylinder $\M$-equivalent to $\Cyl_i$ at all points along the path. 

Now horizontally shear $(X, \omega) + Tv$ until $\Cyl_d$ becomes vertical. The equations on period coordinates cutting out $\M$ may be parallel translated along this path and so we see that along this path no new cylinders become $\M$-equivalent to cylinders in $\Cyl_i$ for any $i$. 
 \end{proof}
 
 In the following proof we will say that two heights (of cylinders) $h_1$ and $h_2$ are $a$-close if $|h_1 - h_2| \leq a$. 

\begin{sslem}\label{cylinderlemma}
Let $\M$ be an affine invariant submanifold. Suppose that for any horizontally periodic $(X, \omega) \in \M$ such that $\Twist \M = \Pres \M$, the heights of any two $\M$-equivalent horizontal cylinders are identical. Then the heights of any two equivalent cylinders on any translation surface in $\M$ are identical.
\end{sslem}
\begin{proof}
The following proof is almost identical to the proof of Theorem 5.1 in Wright~\cite{Wcyl}. Let $\mathcal{C}$ be an equivalence class of horizontal cylinders on $(X_0, \omega_0) \in \mathcal{M}$. Let $\epsilon > 0$ be taken to be smaller than the heights of all cylinders in $\mathcal{C}$. Consider the following iterative process:
\begin{enumerate}
\item If $(X_i, \omega_i)$ is horizontally periodic and $\mathrm{Twist}_{(X_i, \omega_i)} = \mathrm{Pres}_{(X_i, \omega_i)}$ then end the process. Otherwise, Smillie-Weiss~\cite{SW2} (Corollary 6) implies that the horocycle flow accumulates on a horizontally periodic translation surface $(Y_i, \eta_i)$. Since the horocycle flow of $(X_i, \omega_i)$ becomes arbitrarily close to $(Y_i, \eta_i)$ we may assume that there is some $T$ so that there is a perturbation of $u_T (X_i, \omega_i)$ - through surfaces in $\M$ - to $(Y_i, \eta_i)$ so that the cylinders on $(X_i, \omega_i)$ persist on $(Y_i, \eta_i)$ and have heights that are $\frac{\epsilon}{2 \cdot ( g + |\Sigma| - 1 )}$-close to their height on $(X_i, \omega_i)$. By definition of $\M$-equivalence class, if $\M$-equivalent cylinders persist under a perturbation of a translation surface through translation surfaces in $\M$ then they remain $\M$-equivalent. Therefore, the cylinders in $\mathcal{C}$ persist on $(Y_i, \eta_i)$, remain $\M$-equivalent there, and have heights that are $\frac{\epsilon}{2 \cdot ( g + |\Sigma| - 1 )}$ close to their original height. 
\item If $\mathrm{Twist}_{(Y_i, \eta_i)} = \mathrm{Pres}_{(Y_i, \eta_i)}$ then end the process.  Otherwise there is an element $v$ in the cylinder preserving space that fails to be in the twist space. Flowing in the $\sqrt{-1} \cdot v$ direction for an arbitrarily small positive time leads to a surface $(X_{i+1}, \omega_{i+1})$ on which the cylinders in $\Cyl$ persist, are $\M$-equivalent and, have heights that are $\frac{\epsilon}{2 \cdot \left( g+|\Sigma|-1\right)}$-close to their heights on $(Y_i, \eta_i)$; but where the horizontal cylinders from $(Y_i, \eta_i)$ although they persist, do not cover $(X_{i+1}, \omega_{i+1})$. Now return to step 1.
\end{enumerate}
Since the number of cylinders increases with each iteration and the largest possible number of horizontal cylinders is $g+|\Sigma|-1$ the process terminates after at most $g+|\Sigma|-1$ cycles. Each iteration alters the height of each cylinder in $\mathcal{C}$ by at most $\frac{\epsilon}{g + |\Sigma|-1}$. Since there are at most $g+|\Sigma|-1$ iterations, when the process terminates each cylinder has had its height altered by at most $\epsilon$. Moreover, at the end of the process, $\Cyl$ is a collection of $\M$-equivalent cylinders on a translation surface $(X, \omega)$ with heights $\epsilon$-close to their original heights and where $\Twist \M = \Pres \M$. By hypothesis, these cylinders on $(X, \omega)$ have identical heights. Therefore, the cylinders in $\mathcal{C}$ on $(X_0, \omega_0)$ all have heights that are $\epsilon$-close to one another for arbitrarily small $\epsilon$. 
 \end{proof}
 
%
%

\section{Odd Dimensional Orbit Closures}\label{ODOC}

Throughout this section $\M$ will be an affine invariant submanifold in $\hyp(g-1,g-1)$ of odd complex dimension $2r+1$. The two main results of this section are the following:

\begin{ssthm}\label{T:Odd1}
If $\M$ is higher rank and $(X, \omega)$ is a horizontally periodic translation surface with the twist space and cylinder preserving space coinciding, then $(X, \omega)$ has $g+1$ horizontal cylinders and equivalent horizontal cylinders are nonadjacent and have identical heights.
\end{ssthm}

\noindent This theorem will be the key to showing that if $\M$ is higher rank and odd dimensional then it is a branched covering of $\hyp(r-1,r-1)$. Before stating the second result we associate to a collection of horizontal cylinders $\cC$ the deformation $\displaystyle{\sigma_{\cC} =  \sum_{c \in \cC} h_c \gamma_c^* }$ where $h_c$ is the height of the horizontal cylinder $c$, $\gamma_c$ is the core curve of $c$ (oriented from left to right), and $\gamma_c^*$ is the dual cohomology class corresponding to $\gamma_c$. The second main result of this section is the following.

\begin{ssthm}\label{T:Odd2}
If $\M$ is rank one and $(X, \omega)$ is horizontally periodic and the twist space is spanned by $\sigma_{\cC_1}$ and $\sigma_{\cC_2}$ where $\cC_1$ and $\cC_2$ are disjoint collections of cylinders, then if two cylinders belong to the same $\cC_i$, for $i = 1$ or $2$, they have the same height and are not adjacent.
\end{ssthm}

\noindent The application of Theorem~\ref{T:Odd2} to the study of higher rank affine invariant submanifolds is less obvious than the application of Theorem~\ref{T:Odd1}. After all, Theorem~\ref{T:Odd1} is expressly about higher rank affine invariant submanifolds whereas Theorem~\ref{T:Odd2} is expressly about rank one affine invariant submanifolds. The connection to our problem is that we will classify the rank two rel zero affine invariant submanifolds by degenerating to a rank one affine invariant submanifold whose twist space is as described in Theorem~\ref{T:Odd2}.

We now fix some notation that will be used in the sequel. Let $p: T_{(X, \omega)} \M \ra H^1(X, \C)$ be the projection of the tangent space of $\M$ at $(X, \omega)$ onto absolute cohomology. Let $\eta$ be a nonzero relative cohomology class contained in $\ker(p) \cap H^1(X, \Sigma; \R)$ where $\Sigma$ is the zero set of $\omega$. In other words, $\eta$ is a relative deformation on $(X, \omega)$ that preserves the horizontal cylinders.

\begin{ssthm}\label{reldef}
Let $(X, \omega)$ be a translation surface in $\M$ with at least one horizontal cylinder. The following are equivalent:
\begin{enumerate}
\item $(X, \omega)$ has $g+1$ horizontal cylinders.
\item $\Twist \M = \Pres \M$
\item The relative deformation $\eta$ is contained in the twist space.
\item The Lindsey tree is a tree and not just a half-tree.
\end{enumerate}
When any of these equivalent conditions holds label the cylinders $\{c_0, \hdots, c_g\}$ and the core curves $\{\gamma_0, \hdots, \gamma_g\}$, it follows that $\ds{ \eta = \sum_{i=0}^g (-1)^{d(c_0, c_i)} \gamma_i^*}$ where $d(c_i, c_0)$ is the distance between $c_i$ and $c_0$ in the Lindsey tree. 
\end{ssthm}
\begin{proof}
($1 \Rightarrow 2$) By Theorem~\ref{TP} if $(X, \omega)$ has $g+1$ horizontal cylinders then $\Twist \M = \Pres \M$. 

($2 \Rightarrow 3$) Since any relative deformation fixes the core curves of every cylinder it follows that if $\Twist \M = \Pres \M$ then $\eta$ is contained in $\Twist \M$. 

($3 \Rightarrow 1$) Now suppose that $\eta$ is contained in the twist space, i.e $\ds{ \eta = \sum_c a_c \gamma_c^* }$ where the sum is taken over the collection of horizontal cylinders and $a_c \in \R$.  Since $(X, \omega)$ is a translation surface in a hyperelliptic connected component whenever $v$ and $v'$ are adjacent cylinders there is an absolute period contained in $v \cup v'$ that intersects $\gamma_v$ and $\gamma_{v'}$ exactly once with the same orientation. This follows since any two adjacent cylinders are adjacent on both boundaries. Since this period must be unchanged by the relative deformation $\eta$ it follows that $a_v + a_{v'} = 0$ for any two adjacent cylinders $v$ and $v'$. Since the vertices are arranged in a tree we have that up to scaling the purely relative deformation is $\ds{ \sum_{c} (-1)^{d(c, c_0)} \gamma_c^*}$ where $c_0$ is some fixed cylinder. Finally, the Lindsey tree of $(X, \omega)$ cannot have any half-edges since they yield nonzero elements of absolute homology that are supported in a single cylinder and hence will be altered by $\ds{ \sum_{c} (-1)^{d(c, c_0)} \gamma_c^* }$. Therefore, the Lindsey tree of $(X, \omega)$ is a tree (not just a half-tree) and $(X, \omega)$ has $g+1$ horizontal cylinders.

($1$ if and only if $4$) $(X, \omega)$ has $g+1$ horizontal cylinders if and only if there are $g$ full edges in the Lindsey tree (equivalently $2g$ half-edges with each full edge counted as two half-edges). Since the total number of half edges for a Lindsey tree corresponding to a surface in $\hyp(g-1,g-1)$ is $2g$ there are $g$ full edges if and only if there are no half-edges. 

 \end{proof}

\begin{ssthm}\label{rel=1}
If $\M$ has rank $r>1$ and if $(X, \omega) \in \M$ has $g+1$ horizontal cylinders then $(X, \omega)$ has $r+1$ equivalence classes. If $\Cyl_0, \hdots, \Cyl_r$ is an enumeration of the equivalence classes then
\[ \Twist \M = \mathrm{span}_\R \{ u_{\Cyl_0}, \hdots, u_{\Cyl_r} \} \]
where $u_{\Cyl_i}$ is the standard shear of $\Cyl_i$. Any two cylinders in the same equivalence class have identical heights and are an even distance apart in the Lindsey tree. Moreover, $\M$ is defined over $\Q$.
\end{ssthm}
\begin{proof}
If $(X, \omega)$ has $g+1$ cylinders then $\Twist \M = \Pres \M$. Since $\M$ is higher rank it follows that there are at least two $\M$-equivalence classes of horizontal cylinders and so Theorem~\ref{reldef} implies that $\eta$ is not supported on a single $\M$-equivalence class. By the twist space decomposition theorem (Theorem~\ref{keep}) it follows that the only element of the twist space supported on a single equivalence class is the standard shear and hence $\eta$ is a real linear combination of standard shears, i.e. 
\[ \sum_{i=0}^g (-1)^{d(c_i, c_0)} \gamma_i^* = \sum_{i=1}^n a_i \sum_{c \in \mathcal{C}_i} h_c \gamma_c^* \]
Since $\gamma_c^*$ are all linearly independent in $T_{(X, \omega)} \strata$ where $\strata$ is the component of the stratum containing $(X, \omega)$ it follows that any two equivalent cylinders have the same height and are an even distance apart in the Lindsey tree. Finally we see that $\Twist \M$ is spanned by standard shears and its projection to absolute cohomology has a one-dimensional kernel. It follows that $(X, \omega)$ has $r+1$ equivalence classes of cylinders.

 By Theorem 7.1~\cite{Wcyl} to show that $\M$ is defined over $\Q$ it suffices to show that the ratio of lengths of core curves of any two equivalent horizontal cylinders is always rational. Notice that the only element of the twist space supported on a single equivalence class is the standard shear. This implies that the ratio of moduli of any two equivalent cylinders is rational since otherwise there would be a deformation supported on the equivalence class and contained in the tangent space of $\M$ that was not the standard shear. Since the heights of any two equivalent cylinders are identical the result follows.
 \end{proof}

\begin{sscoro}\label{same-heights}
Any two equivalent horizontal cylinders on any translation surface in $\M$ have identical heights when $\M$ is higher rank.
\end{sscoro}
\begin{proof}
This is immediate from Lemma~\ref{cylinderlemma} and Theorem~\ref{rel=1}.
 \end{proof}

\begin{sscoro}\label{ss}
If $(X, \omega) \in \M$ is a translation surface with at least one horizontal cylinder and $\M$ is higher rank then the twist space is spanned by standard shears of the $\M$-equivalence classes of horizontal cylinders.
\end{sscoro}
\begin{proof}
By Theorem~\ref{reldef}, Theorem~\ref{rel=1} establishes this result when the twist space contains $\eta$. When the twist space does not contain $\eta$ the result is immediate by the twist space decomposition theorem (Theorem~\ref{keep}).
 \end{proof}

\begin{ssthm}\label{simple}
Suppose that $\M$ has rank one and $(X, \omega)$ has $g+1$ cylinders. Suppose furthermore that $\mathcal{C}_0, \mathcal{C}_1$ is a partition of the cylinders so that $T_{(X, \omega)}^\R \M$ contains $\ds{ \sum_{c \in C_0} h_c \gamma_c^*}$. Then any two cylinders in $\mathcal{C}_0$ (resp. $\mathcal{C}_1$) have identical heights and are an even distance apart in the Lindsey tree. Moreover, $\M$ is defined over $\Q$ and hence is a branched covering construction of $\strata(0,0)$.
\end{ssthm}
\begin{proof}
Since $T_{(X, \omega)}^\R \M$ contains $u_0 := \ds{ \sum_{c \in C_0} h_c \gamma_c^*}$ and contains the standard shear $\ds{ \sum_{c \in C_0 \cup C_1} h_c \gamma_c^*}$, it contains $u_1 := \ds{ \sum_{c \in C_1} h_c \gamma_c^*}$. Since $\M$ is one-dimensional and has one dimension of rel, it follows that the relative deformation $\eta$ is a real linear combination of $u_1$ and $u_2$. The proof is now identical to the proof of Theorem~\ref{rel=1}.
 \end{proof}

%
%

\section{A Partial Compactification of Strata of Abelian Differentials}\label{Boundary1}
A natural partial compactification of a stratum of abelian differentials is the bundle of stable finite volume abelian differentials over the Deligne-Mumford compactification of moduli space. However, it is often more natural from the perspective of flat geometry to consider a quotient of this space that ignores components of the underlying curve on which the stable one-form vanishes. For the remainder of the paper this quotient will be called the partial compactification of a stratum. It was introduced in McMullen~\cite{McNav} and extensively studied in Mirzakhani-Wright~\cite{MirWri}.

One aspect of this partial compactification is that boundary translation surfaces may have marked points. In the following, if we specify a boundary translation surface $(X, \omega, \Sigma)$ then we understand that $X$ may be a disjoint union of Riemann surfaces, $\omega$ is a holomorphic one form that does not have zero area on any component of $X$, and $\Sigma$ is a collection of marked points that includes all the zeros of $\omega$ as well as potentially new marked points that arise. One of the difficulties that we will tackle in the next section is ensuring that marked points do not arise.

The following example of convergence to the boundary is Example 3.1 of~\cite{MirWri}. Suppose that $\M$ is an affine invariant submanifold and let $(X, \omega)$ be a translation surface in $\M$ with an $\M$-equivalence class of horizontal cylinders $\Cyl$. Suppose that $\Cyl$ does not cover $(X, \omega)$ and that the union of cylinders in $\Cyl$ contains a vertical saddle connection. Let $(X_t, \omega_t)$ be $(X, \omega)$ with $\Cyl$ vertically shrunk by $e^t$. By the cylinder deformation theorem it follows that $(X_t, \omega_t)$ is a smooth path in $\M$. This sequence converges to a translation surface $(X_\infty, \omega_\infty)$ on the boundary of $\M$. To form $(X_\infty, \omega_\infty)$ use the following procedure:

\begin{enumerate}
\item Delete every cylinder in $\Cyl$ from $(X, \omega)$ to form a translation surface with boundary. The boundary is a collection of saddle connections. 
\item If there is a point $p$ on the boundary of $(X, \omega) - \Cyl$ that is joined to a zero or marked point of $(X, \omega)$ by a vertical line that is completely contained in $\Cyl$ then mark $p$. On the boundary translation surface these points will be either marked points or zeros of the boundary holomorphic one-form. Adding in these marked points may divide saddle connections into several smaller saddle connections.
\item If two saddle connections on the boundary of $(X, \omega) - \Cyl$ were connected by a vertical line that was completely contained in $\Cyl$ then glue the two saddle connections together. The resulting translation surface is $(X_\infty, \omega_\infty)$.
\end{enumerate}

This construction is called a horizontal cylinder collapse. The analogous construction with an $\M$-equivalence class of vertical cylinders will be called a vertical cylinder collapse. 

\begin{ssthm}[Mirzakhani-Wright~\cite{MirWri}; Proposition 2.3]
Given a sequence $(X_n, \omega_n, \Sigma_n)$ of translation surfaces converging to $(X, \omega, \Sigma)$ there are collapse maps $f_n: X_n \ra X$ such that 
\begin{enumerate}
\item There is a neighborhood $U_n$ of $\Sigma_n$ so that $f_n: X_n - U_n \ra X$ is a diffeomorphism onto its image with inverse $g_n$.
\item The injectivity radius of $U_n$ goes to zero uniformly in $n$.
\item $g_n^* \omega_n$ converges to $\omega$ in the compact open topology.
\end{enumerate}
Define the space of vanishing cycles to be 
\[ V_n = \ker\left( f_n: H_1(X_n, \Sigma_n; \C) \ra H_1(X, \Sigma; \C) \right) \]
For large enough $n$ this space is constant and will be called $V$.
\end{ssthm}

Returning to the example of the horizontal cylinder collapse: the space of vanishing cycles will be the subspace spanned by the heights of the horizontal cylinders in $\Cyl$ and by any homology classes that have a representative supported in a subsurface that collapses to a point in the limit. By ``height" of a cylinder we mean any saddle connection joining a zero on one boundary of a cylinder to a zero on the other and that intersects the core curve of the cylinder exactly once. 

\begin{ssthm}[Degeneration Theorem; Theorem 2.7, Mirzakhani-Wright~\cite{MirWri}]\label{MirWri}
Let $\M$ be an affine invariant submanifold. Let $(X_n, \omega_n, \Sigma_n)$ be translation surfaces in $\M$ converging to $(X_\infty, \omega_\infty, \Sigma)$. Let $(Y, \eta)$ be a component of $(X_\infty, \omega_\infty)$ and let $\iota: (Y,\eta) \hookrightarrow (X_\infty, \omega_\infty)$ be the inclusion map. Let $V$ be the space of vanishing cycles. The $\GL_2 \R$ orbit closure of $(Y, \eta)$ is an affine invariant submanifold $\M'$ whose tangent space is 
\[ T_{(Y, \eta)} \M' = \iota^* \left( T_{(X_n, \omega_n)} \cap \mathrm{Ann}(V_n) \right) \]
where $T_{(X_n, \omega_n)} \cap \mathrm{Ann}(V_n)$ has been identified with the tangent space at the boundary by parallel transport. As a consequence, $\dim_\C \M' < \dim_\C \M$ and $\mathrm{rk}(\N) \leq \mathrm{rk}(\M)$ where the inequality is strict if $\mathrm{rel}(\M) = 0$
\end{ssthm}

%
%

\section{Degenerating to the Boundary in Hyperelliptic Components of Strata}\label{Boundary2}

The strategy for classifying higher rank orbit closures in hyperelliptic components in this paper is an inductive one. We will study affine invariant submanifolds $\M$ by studying their boundary in the Mirzakhani-Wright partial compactification. However, we immediately confront two potential obstacles to our approach. First, the boundary of $\M$ might contain marked points and second, it might not belong to a hyperelliptic component. The goal of this section is to devise degenerations that avoid these two potential problems.

Let $(X, \omega)$ be a horizontally periodic translation surface in a hyperelliptic component of a stratum of abelian differentials on genus $g > 1$ Riemann surfaces. Suppose that $\cV$ is a collection of vertical cylinders that contains at least one horizontal saddle connection. Suppose too that $\cC$ is a collection of horizontal cylinders that are not self-adjacent and that contains a horizontal saddle connection. Suppose that neither collection of cylinders covers the surface. We would like to collapse these cylinders to pass to a surface on the boundary of the stratum of abelian differentials. To be clear, collapsing a collection of vertical cylinders means collapsing them horizontally, i.e. applying the matrix $\displaystyle{ \begin{pmatrix} e^{-t} & 0 \\ 0 & 1 \end{pmatrix} }$ to the vertical cylinders (while fixing the rest of the translation surface) and taking the limit as $t$ goes to infinity. Similarly, collapsing a collection of horizontal cylinders means vertically collapsing them.

Every degeneration of a translation surface that we use will be collapsing a collection of vertical or horizontal cylinders. In this section, we will show the following

\begin{ssthm}\label{T:two-degenerations}
Collapsing either $\cV$ or $\cC$ degenerates to a disjoint union of translation surfaces in hyperelliptic components of strata of abelian differentials. 
\end{ssthm}

\noindent Recall that in genus one we have defined the hyperelliptic components to be $\mathcal{H}(0)$ and $\mathcal{H}(0,0)$. Moreover, the boundary of $\M$ that is referred to in the theorem is the boundary in the sense of the Mirzakhani-Wright partial compactification. Finally we remark that a collection $S$ of parallel cylinders is said to be self-adjacent if two cylinders in $S$ are adjacent or if there is a single cylinder in $S$ whose two boundaries are glued together along a saddle connection. To fix notation, let $\Gamma$ be the Lindsey tree of $(X, \omega)$ and let $J$ be the hyperelliptic involution on $(X, \omega)$. 

We begin by establishing Theorem~\ref{T:two-degenerations} in the case of vertical collapses of cylinders. The intuition for the result is the following. Corollary~\ref{LTC} tells us, roughly, that whenever we glue together cylinders of hyperelliptic combinatorial type (i.e. ones that look like the cylinders in Figure~\ref{fig:CT}) along a tree that we must get a translation surface in a hyperelliptic component of a stratum. When we collapse a collection of vertical cylinders on a horizontally periodic translation surface, the horizontal cylinders persist on the boundary; they still have hyperelliptic combinatorial type; and they are still glued together in a disjoint union of trees. So the boundary translation surface must be a disjoint union of translation surfaces in hyperelliptic components of strata.

\begin{sslem}[Vertical Cylinder Collapse Lemma]\label{collapse0}
Collapsing $\cV$ degenerates the translation surface to a disjoint union of translation surfaces in hyperelliptic components of strata of abelian differentials. 
\end{sslem}
\begin{proof}
By Corollary~\ref{LTC} a translation surface belongs to a hyperelliptic component if and only if it is constructed in the following way:
\begin{enumerate}
\item Fix a tree with a cyclic ordering around each vertex. For each degree $n$ vertex in the tree associate a horizontal cylinder of hyperelliptic combinatorial, i.e. the cylinder shown in Figure~\ref{fig:CT} up to changing the lengths of the saddle connections and horizontally shearing.
\item When two vertices are joined along an edge, open up the corresponding edges on the appropriate horizontal cylinders and glue the two cylinders together.  
\end{enumerate}
This provides both blueprints on how to build translation surfaces in hyperelliptic components and a certificate that a surface belongs to a hyperelliptic component. 

Recall that any cylinder in $(X, \omega)$ is invariant under the hyperelliptic involution. Therefore, if the proportion of a horizontal saddle connection $s$ contained in a cylinder $V$ is $p$, then the proportion of the saddle connection $J(s)$ contained in $V$ is also $p$. Collapsing a collection of vertical cylinders $\cV$ passes to a boundary translation surface $(Y, \eta)$ that can be constructed from $(X, \omega)$ in the following way.

\begin{enumerate}
\item Let $\Gamma'$ be the tree that is formed from $\Gamma$ when all the edges corresponding to saddle connections completely contained in $\cV$ are deleted. If a node has no edges attached to it in $\Gamma'$, then delete it. 
\item The remaining nodes correspond to horizontal cylinders on $(X, \omega)$ that persist on the boundary translation surface $(Y, \eta)$. To change a horizontal cylinder $C$ on $(X, \omega)$ to the corresponding one on $(Y, \eta)$ take each saddle connection $s$ on the boundary $C$ and change its length to the length of $s$ not contained in $\cV$. The new cylinder on $(Y, \eta)$ still has hyperelliptic combinatorial type. 
\end{enumerate}

Since $(Y, \eta)$ can be constructed by gluing together horizontal cylinders of hyperelliptic combinatorial type along a disjoint union of trees, Corollary~\ref{LTC} implies that $(Y, \eta)$ is a disjoint union of translation surfaces in hyperelliptic components of strata of abelian differentials. 
 \end{proof}

Now we will complete the proof of Theorem~\ref{T:two-degenerations} by analyzing degenerations that involve collapsing a collection of horizontal cylinders. 

\begin{sslem}[Horizontal Cylinder Collapse Lemma]\label{collapse1}
Collapsing $\cC$ degenerates to a disjoint union of translation surfaces in hyperelliptic components of strata. 
\end{sslem}
\begin{proof}
Recall that the boundary translation surface $(X_\infty, \omega_\infty)$ may be constructed from $(X, \omega)$ in the following way:
\begin{enumerate}
\item Delete every cylinder in $\Cyl$ from $(X, \omega)$. The result is a translation surface with boundary where the boundary consists of saddle connections that formerly bordered cylinders in $\Cyl$.
\item For each saddle connection on the boundary of $(X, \omega) - \Cyl$ add a marked point to the saddle connection for each point $p$ such that the vertical line contained in $\Cyl$ passing through $p$ terminates at a zero of $\omega$. Since the newly added marked points are invariant under the hyperelliptic involution each cylinder on $(X, \omega) - \Cyl$ continues to have hyperelliptic combinatorial type.
\item Glue together saddle connections on the boundary of $(X, \omega) - \Cyl$ which were connected by a vertical line contained in $\Cyl$. This saddle connection identification is again invariant under the hyperelliptic involution. Let $\Gamma'$ be $\Gamma$ with vertices in $\Cyl$ deleted, edges connected to $\Cyl$ deleted, and new edges added between two cylinders that are connected by a vertical line in $\Cyl$.
\end{enumerate}
By Corollary~\ref{LTC}it remains to verify that the cylinder diagram $\Gamma'$ is a tree. Notice that $\Gamma'$ is constructed by deleting each vertex $v$ in $\Cyl$ and adding in edges between vertices that were adjacent to $v$. To show that $\Gamma'$ is a tree it suffices to show that whenever a vertex $v$ is deleted no cycle forms among the vertices that were formerly adjacent to $v$. 

Rephrased, it suffices to show the following. Suppose that $C$ is a single cylinder of hyperelliptic type. Let $s_1, \hdots, s_n$ be saddle connections on the boundary of $C$. Let $G$ be a graph with $n$ vertices labelled $\{1, \hdots, n\}$. Connect vertices $i$ and $j$ if $s_i$ and $J(s_j)$ are connected by a vertical line. Then $G$ is a disjoint union of trees. This follows immediately from the following lemma.
\begin{sublem}
Suppose that there is a graph $G$ with vertices labelled $\{1, \hdots, n\}$, which we imagine as being cyclically ordered. Let $C_i$ be the set of vertices connected to vertex $i$. Suppose that for all $i$ there is an increasing subset of $\{1, \hdots, n\}$ (perhaps wrapping around $0$) that we will denote $I_i = (k_i, k_i +1, \hdots, \ell_i)$ such that 
\begin{enumerate}
\item $C_i \subseteq I_i$ for all $i$.
\item $I_i \cap I_{i+1} = \{k_i \} = \{ \ell_{i+1} \}$ if $n>2$.
\end{enumerate}
Then $G$ is a disjoint union of trees. 
\end{sublem}
\begin{proof}
Proceed by induction on $n$. The $n=2$ base case is trivial. Now suppose that $n>2$. Suppose to a contradiction that $G$ contains a cycle. Let $\gamma$ be the shortest cycle in $G$. If the cycle fails to contain every vertex, then delete the vertices not contained in $\gamma$ from $G$. The induction hypothesis implies that the resulting graph cannot contain a cycle, which is a contradiction. Suppose then without loss of generality that $\gamma$ involves every vertex. 

If the degree of a vertex $i$ is greater than two then we may suppose that $C_i = \{k_i, k_{i+1}, \hdots, \ell_i\}$. By the hypotheses, vertices $k_i+1, \hdots, \ell_i -1$ only connect to vertex $i$. Since $\gamma$ is the shortest cycle in $G$ it does not pass through vertices $k_i + 1, \hdots, \ell_i-1$ contrary to our assumption that $\gamma$ passes through every vertex. It follows that every vertex in $G$ has degree two.

Since every vertex in $G$ has degree two and appears exactly once in $\gamma$ it follows that $C_i = \{k_i, k_i+1\}$ and $C_{i+1} = \{k_i-1, k_i\}$ for all $i$. Therefore, the path $\gamma$ is $(\gamma_1, \gamma_2, \hdots, \gamma_n) = (1, k_1, 2, k_1-1, 3, k_1-2, \hdots)$. Notice that the order of the odd vertices is  $(\gamma_1, \gamma_3, \hdots) = (1, 2, \hdots, n)$. Since every vertex appears exactly once in $\gamma$ this implies that $n=2m+1$ and the path $\gamma$ is $(1, m+2, 2, m+3, \hdots, m, 2m+1, m+1)$. However the order of the even vertices must be $(\gamma_2, \gamma_4, \hdots) = (m+2, m+1, m, \hdots)$. Therefore the cyclic order $(1,2, \hdots, n)$ and the cyclic order $(n, n-1, \hdots, 1)$ must be the same order. This only occurs when $n=2$. But we have supposed that $n>2$, which is a contradiction.
 \end{proof}
\text{}
 \end{proof}

As mentioned earlier all of the degenerations that we will use in this paper will be either a horizontal or a vertical cylinder collapse. So let's analyze this situation. Let $\M$ be a higher rank affine invariant submanifold in a hyperelliptic component of a stratum. Let $(X, \omega)$ be a horizontally periodic translation surface. 

\begin{sslem}\label{L:X-twist}
If $(X, \omega)$ is horizontally periodic then its twist space is spanned by standard shears
\end{sslem}
\begin{proof}
By Theorem~\ref{keep} if this is not the case then there is an $\M$-equivalence class of cylinders $\cC$ and a twist space deformation supported on $\cC$ that is rel. However, by Theorem~\ref{reldef} this is only possible if $\cC$ contains $g+1$ horizontal cylinders. However, $(X, \omega)$ contains at most $g+1$ horizontal cylinders and if all of them belong to one equivalence class then $(X, \omega)$ belongs to a rank one orbit closure, which contradicts the hypothesis that $\M$ is higher rank.
\end{proof}

Let $\cC$ be an equivalence class of either vertical cylinders or horizontal cylinders that do not form a self-adjacent equivalence class on $(X, \omega)$. Let $(Y, \eta)$ be the translation surface formed by collapsing $\cC$ and let $(Z, \zeta)$ be a component of $(Y, \eta)$. Let $\N$ be the affine invariant submanifold in the boundary of $\M$ that contains $(Z, \zeta)$. 

\begin{sslem}[Twist Space Degeneration Lemma]\label{L:tsdl}
Let $k$ be the number of pairwise $\M$-inequivalent horizontal cylinders that persist on $(Z, \zeta)$. If $k \geq 2$, the following hold:
\begin{enumerate}
\item The dimension of the twist space of $(Z, \zeta)$ in $\N$ is $k$.
\item If $\N$ is higher rank then two horizontal cylinders in $(Z, \zeta)$ are $\N$-equivalent if and only if their preimages on $(X, \omega)$ were $\M$-equivalent.
\item If $\N$ is rank one but $(Z, \zeta)$ contains two cylinders from distinct $\M$-equivalence classes $\cC_1$ and $\cC_2$ on $(X, \omega)$, then no two cylinders from $\cC_i$ are adjacent on $(Z, \zeta)$ and any two such cylinders have identical height for $i = 1, 2$.
\item If the twist space and cylinder preserving space coincide on $(X, \omega)$ and a cylinder from every equivalence class persists on $(Z, \zeta)$, then $\M$ is even-dimensional, $\N$ is odd-dimensional, and any two equivalent cylinders that persist on $(Z, \zeta)$ are not adjacent and have identical heights.
\end{enumerate}
\end{sslem}
For the first claim, let $\cC$ be a maximal collection of horizontal cylinders on $(Z, \zeta)$ that were equivalent on $(X, \omega)$. Let $u_\cC$ be the standard shear of these cylinders on $(Z, \zeta)$. By the degeneration theorem of Mirzakhani-Wright~\cite{MirWri} (Theorem~\ref{MirWri} in this paper, Theorem 2.7 in \cite{MirWri}) each $u_{\cC}$ is a tangent vector on $(Z, \zeta)$. For any collection of cylinders $C$, the standard shear $u_C$ belongs to the twist space. Since the equivalence classes are pairwise disjoint, the standard shears they induce on $(Z, \zeta)$ all belong to the twist space and so the twist space is at least $k$ dimensional. We will complete the proof of the first claim after the proof of the third claim.

For the second claim, suppose that $\N$ is higher rank. Any two $\M$-equivalent cylinders that persist on the boundary remain $\N$-equivalent since the colinearity of their core curves is an algebraic equation that extends to the boundary. Suppose that $\cC_1, \hdots, \cC_n$ are the equivalence classes of cylinders that persist on $\N$. Any $\N$-equivalence class of horizontal cylinders must $\bigcup_S \cC_i$ for some $S$ a subset of $\{1, \hdots, n\}$. By Lemma~\ref{L:X-twist} the twist space of $\N$ is spanned by standard shears of equivalence classes of horizontal cylinders. By the proof of the first claim the twist space contains $u_{\cC_1}, \hdots, u_{\cC_n}$. Therefore, $\cC_1, \hdots, \cC_n$ are exactly the $\N$-equivalence classes. In this case we see that the the dimension of the twist space is exactly $k$.

For the third claim, suppose that exactly two equivalence classes of cylinders $\cC_1$ and $\cC_2$ persist on $\N$ and suppose that $\N$ has rank one. By Theorem~\ref{T:Odd2} the cylinders in $\cC_i$ are not adjacent and all have identical heights for each $i = 1, 2$. In this case we also have that the dimension of the twist space is $k$. This completes the proof of the third claim.

For the fourth claim, suppose that the twist space and cylinder preserving space coincide on $(X, \omega)$ and that a cylinder from every equivalence class persists on $(Z, \zeta)$. By the first claim, the dimension of the twist space of $(Z, \zeta)$ in $\N$ and the dimension of the twist space of $(X, \omega)$ in $\M$ are identical. The assumption that the twist space and cylinder preserving space coincide on $(X, \omega)$ guarantees that the twist space on $(X, \omega)$ is maximal dimensional and hence has dimension $\mathrm{rank}(\M) + \mathrm{rel}(\M)$. By the degeneration theorem of Mirzakhani-Wright~\cite{MirWri} (Theorem~\ref{MirWri} in this paper, Theorem 2.7 in \cite{MirWri}), 
\[ \mathrm{rank}(\N) \leq \mathrm{rank}(\M) \]
\noindent with strict inequality when $\mathrm{rel}(\M) = 0$. Since $\N$ belongs to a hyperelliptic component of a stratum $\mathrm{rel}(\N) \leq 1$. This observation implies that 
\[ \mathrm{rank}(\N) + \mathrm{rel}(\N) = \mathrm{rank}(\M) + \mathrm{rel}(\M) \]
This expression implies that 
\[ \dim_\C \N = 2 \cdot \mathrm{rank}(\N) + \mathrm{rel}(\N) \leq 2 \cdot \mathrm{rank}(\M) + \mathrm{rel}(\M) = \dim_\C \M \]
The degeneration theorem of Mirzakhani-Wright states that the dimension of $\N$ is strictly less than the rank of $\M$ so we see that the rank of $\N$ is strictly less than the rank of $\M$. Since both $(X, \omega)$ and $(Z, \zeta)$ have the same dimensional twist space we also have that $\mathrm{rel}(\N) = 1$. Since the twist space on $(X, \omega)$ is maximal dimensional we have that $\mathrm{rel}(\M) = 0$. Finally since $\N$ is odd dimensional and since the twist space and cylinder preserving space coincide on $(Z, \zeta)$ - since the twist space is maximal dimensional - Theorem~\ref{T:Odd1} and Theorem~\ref{T:Odd2} implies that any two equivalent cylinders that persist on $(Z, \zeta)$ are not adjacent and have identical heights.

%
%

\section{Rank Two Rel Zero Orbit Closures}\label{S:dim4}

Let $\M$ be a rank two rel zero affine invariant submanifold. The goal of this section will be to show that if $\M$ is contained in a hyperelliptic component of a stratum of abelian differentials then it is a branched covering construction of $\mathcal{H}(2)$. There are many reasons to single out this case. First, this case is the basis of our induction argument. Second, the proof that $\M$ is a branched covering is almost identical to the general case, but with fewer technical problems (so it makes the main ideas of the proof more transparent). Finally, the proof relies on a lemma, which we call the ``Prototype Lemma", that has found application in several forthcoming results in flat geometry. Alex Wright suggested the formulation and proof of the Prototype Lemma. 

\begin{sslem}\label{prototype-lemma}
Let $(X, \omega)$ be a translation surface in an rank two rel zero affine invariant submanifold $\M$ that belongs to any component of any stratum of abelian differentials. Suppose that $(X, \omega)$ contains two non-intersecting $\M$-equivalence classes of cylinders $\cC_1$ and $\cC_2$. If $\cC_1$ contains a saddle connection parallel to the core curves of the cylinders in $\cC_2$, then $(X, \omega)$ is periodic in that direction.
\end{sslem}
\begin{proof}
Since $\M$ is rank two rel zero if there are two distinct equivalence classes of parallel cylinders on $(X, \omega)$ in the $v$-direction, then $\M$ is periodic in the $v$-direction. Therefore, since the statement is immediate when $\cC_1$ and $\cC_2$ are parallel, let's assume that they are not parallel. After rotating and shearing we may assume without loss of generality that $\cC_1$ is a collection of horizontal cylinders and that $\cC_2$ is a collection of vertical ones. The condition that $\cC_1$ contains a saddle connection with period $v_2$ now becomes that $\cC_1$ contains a vertical saddle connection. Let $(X_\infty, \omega_\infty)$ be the boundary translation surface formed by collapsing $\cC_1$. 

By the degeneration theorem of Mirzakhani-Wright~\cite{MirWri} (Theorem~\ref{MirWri} in this paper, Theorem 2.7 in \cite{MirWri}), if $\cC_1$ is collapsed and $(Y, \eta)$ is any component of the boundary translation surface $(X_\infty, \omega_\infty)$, then the orbit closure of $(Y, \eta)$ has complex dimension at most three. In particular, each component of the boundary translation surface is completely periodic. A completely periodic translation surface is characterized by the property that if there is one cylinder in a given direction, then that direction is periodic. Therefore, if we can show that some cylinder from $\cC_2$ appears on each component of the boundary translation surface then we may conclude that every component of the boundary translation surface is vertically periodic. In particular, since $(X_\infty, \omega_\infty)$ was formed by vertically collapsing a collection of cylinders, each component of $(X_\infty, \omega_\infty)$ is vertically periodic if and only if  $(X, \omega)$ is vertically periodic. To summarize, it suffices to show that a cylinder from $\cC_2$ appears on each component of $(X_\infty, \omega_\infty)$. 

Let $C$ be a connected component of the translation surface with boundary formed from $(X, \omega)$ once $\cC_1$ and $\cC_2$ are removed. It suffices to show that each region $C$ borders a cylinders in $\cC_2$. Suppose to a contradiction that $C$ is such a region that does not border a cylinder in $\cC_2$. It follows that the region $C$ borders a horizontal cylinder $D$ in $\cC_1$. Applying the cylinder deformation theorem, we may vertically shear the cylinders in $\cC_2$ so that it contains a horizontal saddle connection while fixing the rest of the translation surface. Vertically shearing the cylinders in $\cC_2$ does not alter the fact that the region $C$ is not adjacent to a cylinder in $\cC_2$. Let $(Z, \zeta)$ be the boundary translation surface formed by collapsing $\cC_2$.

Since $C$ does not border a vertical cylinder in $\cC_2$ it persists isometrically on the boundary translation surface $(Z,\zeta)$. Since it is adjacent to the horizontal cylinder $D$, the region $C$ and the cylinder $D$ remain adjacent on $(Z,\zeta)$ and hence belong to the same component of the boundary translation surface. Since each component of the boundary translation surface is completely periodic and since the component containing the region $C$ also contains the horizontal cylinder $D$, it follows that $C$ is a union of horizontal cylinders. 

Since $C$ did not border a cylinder in $\cC_2$, it was unaffected by the degeneration and hence the region $C$ on $(X, \omega)$ is also covered by horizontal cylinders. By assumption, since $C$ is in the complement of the cylinders in $\cC_1$, the region $C$ is covered by horizontal cylinders that are inequivalent to the horizontal cylinders in $\cC_1$. Moreover, none of the horizontal cylinders in the region $C$ intersect cylinders in $\cC_2$. Therefore, $(X, \omega)$ contains two equivalence classes of horizontal cylinders, neither of which intersects cylinders in $\cC_2$. However, since $\M$ is rank two, if $(X, \omega)$ contains two equivalence classes of horizontal cylinders then these two equivalence classes cover all of $(X, \omega)$ and so some horizontal cylinder must intersect a cylinder in $\cC_2$, which is a contradiction. 
\end{proof} 

\begin{sslem}[Prototype Lemma]\label{prototype}
Any rank two rel zero affine invariant submanifold $\M$ in any component of any stratum contains a translation surface that has exactly two horizontal and two vertical $\M$-equivalence classes of cylinders, so that one of the horizontal $\M$-equivalence classes does not intersect one of the vertical $\M$-equivalence classes.
\end{sslem}
\begin{proof}
Let $\M$ be a rank two rel zero affine invariant submanifold. As described in the cylinder deformation section (Section~\ref{CD} Theorem~\ref{TP}), by Wright Lemma 8.6 \cite{Wcyl} there is a horizontally periodic translation surface $(X, \omega)$ in $\M$ on which the twist space and the cylinder preserving space coincide. Since we have assumed that $\M$ has rank two and no rel, this implies that $(X, \omega)$ has exactly two equivalence classes of horizontal cylinders. Call these equivalence classes $\cC_1$ and $\cC_2$ and let $\sigma_1$ and $\sigma_2$ be the two corresponding standard shears (defined immediately before the statement of Theorem~\ref{TP}). The absence of rel implies that the projections to absolute cohomology of $\sigma_1$ and $\sigma_2$ span a two dimensional subspace. 

The perturbation theorem (Theorem~\ref{perturb} in this paper; Lemma 5.5 in Mirzakhani-Wright~\cite{MirWri}) states that we can find a path in $\M$ that deforms $(X, \omega)$ to a new translation surface $(Y, \eta)$ so that along the path (1) no cylinder in $\cC_1$ or $\cC_2$ vanishes, (2) if $C_1$ is a cylinder equivalent to a cylinder in $\cC_1$ and $C_2$ is cylinder equivalent to a cylinder in $\cC_2$ then $C_1$ and $C_2$ are disjoint, and (3) on $(Y, \eta)$ the cylinders in $\cC_1$ remain horizontal and the cylinders in $\cC_2$ are vertical. After applying the cylinder deformation theorem we may ensure, perhaps after horizontally shearing the cylinders in $\cC_1$ and vertically shearing the cylinders in $\cC_2$, that $\cC_1$ contains a vertical saddle connection and that $\cC_1$ contains a horizontal saddle connection. By Lemma~\ref{prototype-lemma}, the new translation surface is vertically and horizontally periodic and contains a equivalence class of vertical cylinders $\cC_2$ that does not intersect a horizontal equivalence class of cylinders $\cC_1$. Since the translation surface is horizontally and vertically periodic and contains a non-intersecting horizontal and vertical equivalence class of cylinders, it follows that there are at least two horizontal and two vertical equivalence classes. Since $\M$ is rank two rel zero, there are no more than two equivalence classes of cylinders in any given direction and so there are exactly two horizontal and two vertical equivalence classes of cylinders.
 \end{proof}
 
Given a rank two rel zero affine invariant submanifold $\M$, we say that $(X, \omega)$ is a prototype translation surface if $(X, \omega)$ belongs to $\M$, is vertically and horizontally periodic, has a non-intersecting horizontal and vertical equivalence class of cylinders. Given a horizontally periodic translation surface with Lindsey tree $\Gamma$ in a hyperelliptic component of a stratum, we will say that two edges or half-edges of $\Gamma$ are $\M$-equivalent if they connect the same two $\M$-equivalence classes of horizontal cylinders. Half-edges will be understood to connect an $\M$-equivalence class to itself.

From now on we assume that $\M$ is an affine invariant submanifold in a hyperelliptic component of a stratum of abelian differentials and that it is rank two rel zero. Let $(X, \omega)$ be a prototype surface on $\M$ with horizontal equivalence classes $\cC_1$ and $\cC_2$ and vertical equivalence classes $\cV_1$ and $\cV_2$. We assume without loss of generality that $\cC_1$ and $\cV_2$ do not intersect and that moreover they contain a vertical and horizontal saddle connection respectively. Finally, let $\cS_1$ be the equivalence class of horizontal saddle connections that connect a cylinder in $\cC_1$ to a cylinder in $\cC_2$. Let $\cS_1$ be the equivalence class of horizontal saddle connections that connect a cylinder in $\cC_2$ to a cylinder in $\cC_2$. 

\begin{sslem}\label{L:bcch2} 
If $(X, \omega)$ is a prototype surface in $\M$ described above then
\begin{enumerate}
\item Equivalent horizontal cylinders have identical heights.
\item Equivalent horizontal saddle connections have identical lengths.
\item The saddle connections on the boundary of $\cC_2$ alternate between $\cS_1$ and $\cS_2$.
\end{enumerate}
\end{sslem}
\begin{proof}
Given a prototype surface $(X, \omega)$ we can degenerate the surface by collapsing either $\cC_1$ or $\cV_2$. The vertical collapse lemma (Lemma~\ref{collapse0}) implies that the resulting boundary translation surface is a disjoint union of translation surfaces in hyperelliptic components of strata. Let $(X_\infty, \omega_\infty)$ be the boundary translation surface formed by collapsing $\cC_1$. 

Let $(Y, \eta)$ be a component of $(X_\infty, \omega_\infty)$. By the degeneration theorem of Mirzakhani-Wright~\cite{MirWri} (Theorem~\ref{MirWri} in this paper, Theorem 2.7 in \cite{MirWri}), the orbit closure $\N$ of $(Y, \eta)$ is at most three complex-dimensional. Since cylinders from both $\cV_1$ and $\cV_2$ persist on $(Y, \eta)$, the standard shears $\sigma_{\cV_1}$ and $\sigma_{\cV_2}$ are both tangent to $\N$ at $(Y, \eta)$. Since $\N$ is rank one and since the two standard shears are not constant multiples of each other and pair to zero under the cup product pairing, it follows that $\N$ must have nonzero rel. Since $\N$ is at most three complex-dimensional, it follows that $\N$ is rank one rel one. In particular, some nonzero linear combination of $\sigma_{\cV_1}$ and $\sigma_{\cV_2}$ must be rel on $(Y, \eta)$. 

To summarize, any component of $(X_\infty, \omega_\infty)$ has rank one rel one orbit closure and contains $\sigma_{\cV_1}$ and $\sigma_{\cV_2}$ where $\cV_1$ and $\cV_2$ are disjoint collections of vertical cylinders, Since the twist space of a translation surface in a rank one rel one orbit closure is at most two dimensional, it follows that these two tangent vectors span the twist space. We are therefore exactly in the situation described in Theorem~\ref{T:Odd2}. By Theorem~\ref{T:Odd2}, no two cylinders in $\cV_1$ and $\cV_2$ are adjacent on $(X_\infty, \omega_\infty)$ and any two cylinders that both belonged to the same $\cV_i$ for $i = 1$ or $2$ have identical heights. 

Since no two equivalent vertical cylinders can be adjacent in $\cC_2$, claim 3 follows. Moreover, for any horizontal saddle connection on the boundary of a cylinder in $\cC_2$ there is exactly one vertical cylinder that passes through it (otherwise two equivalent cylinders would be adjacent in $\cC_2$). Therefore, the length of each saddle connection in $\cS_i$ is also the height of a cylinder in $\cV_i$ for $i = 1$ and $2$. Claim 2 will follow from claim 1 by symmetry of hypotheses. It remains to establish claim 1, i.e. that equivalent horizontal cylinders have identical heights. 

We begin by showing that any two cylinders in $\cC_1$ have identical heights. We have already shown that if we degenerate $\cV_2$ and two cylinders from $\cC_1$ end up on the same component of the boundary translation surface, then those two cylinders have identical heights. A path $\gamma$ in $X - Z(\omega)$ is called a staircase path if it is piecewise linear and each linear piece is either vertical or horizontal. Let $C$ and $D$ be horizontal cylinders in $\cC_1$. For any staircase path $\gamma$ between $C$ and $D$ it is possible to produce a collection of cylinders $\{ C_i \}_{i=0}^n$ in $\cC_1$ so that $C_0 = C$, $C_n = D$ and so that there are staircase paths $\gamma_i$ from $C_{i-1}$ to $C_i$ that pass through at most one cylinder in $\cV_2$.

The algorithm to produce this new collection of staircase paths is straightforward. Follow the staircase path $\gamma$ until it has entered and exited a cylinder in $\cV_2$. Since no two cylinders in $\cV_2$ are adjacent, before entering and after exiting a cylinder in $\cC_2$, it will be in a cylinder contained in $\cV_1$. Modify $\gamma$ with a vertical line so that once it exits $\cV_2$ it travels vertically and enters a cylinder in $\cC_1$. Now repeat the procedure.

Since any two cylinders may be joined by a staircase path it suffices to show that if $C$ and $D$ are cylinders in $\cC_1$ that are joined by a staircase path $\gamma$ that passes through exactly one cylinder $V$ in $\cV_2$ that $C$ and $D$ have identical heights. Let $s_1$ be the vertical saddle connection on the lefthand boundary of $V$ that $\gamma$ passes through and $s_2$ the one on the righthand boundary. By the cylinder deformation theorem, it is possible to vertically shear $\cV_2$ while fixing the rest of the translation surface so that $s_1$ and $s_2$ are connected by a horizontal line contained in $V$ and so that $\cV_2$ contains some horizontal saddle connection. Collapsing $\cV_2$ ensures that $C$ and $D$ land on the same component of the boundary translation surface and hence, as argued above, that they have identical heights. This establishes that all cylinders in $\cC_1$ have identical heights as desired.

It remains to show that any two cylinders in $\cC_2$ have identical heights. Let $C$ be the cylinder in $\cC_2$ that has smallest height. Let $D$ be any cylinder in $\cC_1$ adjacent to $C$. By the standard position lemma, it is possible to put $C$ and $D$ into standard position. If $V$ is the resulting cylinder then every cylinder equivalent to $V$ must spend the same percent of time in $\cC_2$ as $V$ does by the cylinder proportion theorem. Since every cylinder in $\cC_1$ has identical height and no two are adjacent, it follows that if $V'$ is a vertical cylinder equivalent to $V$ then every cylinder in $\cC_2$ that $V'$ intersects has the same height as $C$. By the cylinder proportion theorem, every cylinder in $\cC_2$ is intersected by a vertical cylinder equivalent to $V$ and so every cylinder in $\cC_2$ has the same height as cylinder $C$.
 \end{proof}
 
We will now rephrase Lemma~\ref{L:bcch2} in a form more obviously connected to branched covering constructions. First, we make a definition.  Suppose that $\{C_1, \hdots, C_n\}$ is a collection of cylinders on a flat surface. We will take the cylinders to be marked at points on their boundary corresponding to cone points of the flat metric. We will say that the cylinders are mutually isogenous if there is a cylinder $C$ with marked boundary and a local isometry $f_i: C_i \ra C$ for each $1 \leq i \leq n$ so that the preimage of the marked points on the boundary of $C$ under $f_i$ are exactly the marked points on the boundary of $C_i$. The upshot of Lemma~\ref{L:bcch2} is the following

\begin{sscoro}
Let $(X, \omega)$ be the prototype surface in $\M$ as in Lemma~\ref{L:bcch2}. Let $\ell_i$ be the length of the saddle connections in $\cS_i$ and let $h_i$ be the heights of the cylinders in $\cC_i$ for $i = 1, 2$. Each cylinder in $\cC_1$ is isogenous to the cylinder in  Figure~\ref{fig:IsogenyH}.
\begin{figure}[H]
\centering
\begin{tikzpicture}
	\draw (0,0) -- (1,0) -- (1,1) -- (0,1) -- (0,0);
	\node at (-.5, .5) {$h_1$}; \node at (.5,1.5) {$\ell_1$};
\end{tikzpicture}
\caption{The cylinder to which all cylinders in $\cC_1$ are isogenous}
\label{fig:IsogenyH}
\end{figure}
\noindent and every cylinder in $\cC_2$ is isogenous to the cylinder in Figure~\ref{fig:IsogenyH'}
\begin{figure}[H]
\centering
\begin{tikzpicture}
	\draw (0,0) -- (2,0) -- (2,1) -- (0,1) -- (0,0);
	\draw[dotted] (1,0) -- (1,1);
	\node at (-.5, .5) {$h_2$}; \node at (.5,1.5) {$\ell_1$}; \node at (1.5, 1.5) {$\ell_2$};
\end{tikzpicture}
\caption{The cylinder to which all cylinders in $\cC_2$ are isogenous}
\label{fig:IsogenyH'}
\end{figure}
\noindent where opposite sides are identified, the labels correspond to lengths, and all angles are right angles.
\end{sscoro}
\begin{proof}
If $C$ is a cylinder in $\cC_1$ then it has height $h_1$ and every saddle connection on the boundary has length $\ell_1$ by Lemma~\ref{L:bcch2}. Since every vertical cylinder that passes through $C$ crosses through a saddle connection exactly once and must fully contain any saddle connection it passes through, it follows that every saddle connection on one boundary of $C$ perfectly vertically aligns with a saddle connection on the other boundary. It follows that $C$ is isogenous to the cylinder in Figure~\ref{fig:IsogenyH}. The proof for cylinders in $\cC_2$ is essentially identical and so we omit it.
\end{proof}

\begin{ssthm}\label{base2}
If $\M$ is a rank two rel zero affine invariant submanifold in a hyperelliptic component of a stratum then it is a branched covering construction of $\strata(2)$.
\end{ssthm}
\begin{proof}
Let $(X, \omega)$ be a prototype surface in $\M$ as in Lemma~\ref{L:bcch2}. We will now construct the surface that $(X, \omega)$ is a translation covering of. Let $(Y, \eta)$ be the translation surface in Figure~\ref{fig:CoveredSurface}.
\begin{figure}[H]
\centering
\begin{tikzpicture}
	\draw (0,0) -- (0,2) -- (1,2) -- (1,1) -- (2,1) -- (2,0) -- (0,0);
	\draw[dotted] (1,0) -- (1,1);
 	\draw[dotted] (0,1) -- (1,1);
	\node at (-.5, .5) {$h_2$}; \node at (.5,2.5) {$\ell_1$};
	\node at (-.5, 1.5) {$h_1$}; \node at (1.5,1.5) {$\ell_2$};
\end{tikzpicture}
\caption{The translation surface that $(X, \omega)$ covers in $\strata(2)$}
\label{fig:CoveredSurface}
\end{figure}
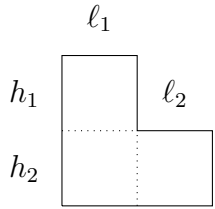
\noindent In Figure~\ref{fig:CoveredSurface} opposite sides are identified, the labels correspond to the lengths of the saddle connections, and all angles are right angles. For each horizontal cylinder on $(X, \omega)$ there is a local isometry that takes it to either the top cylinder, if the horizontal cylinder belonged to $\cC_1$, or the bottom cylinder, if it belonged to $\cC_2$. The maps on the horizontal cylinders of $(X, \omega)$ agree whenever two cylinders share a boundary and so the local isometries glue together to form a map $f: (X, \omega) \ra (Y, \eta)$ that is a translation covering. 

By Corollary~\ref{useful} to show that $\M$ is a branched covering construction of $\mathcal{H}(2)$ it suffices to show that we can find a generic prototype surface. By the cylinder deformation theorem we may suppose without loss of generality that $\ell_1 = \ell_2 = h_1 = 1$ and $h_2 = a$ where $a$ is any transcendental number. Since the moduli of the cylinders in $\cC_1$ and those in $\cC_2$ are not rational multiples of each other, the orbit closure of $(X, \omega)$ contains the standard shears on $\cC_1$ and $\cC_2$. By Avila, Eskin, M\"oller~\cite{AEM}, the projection to absolute cohomology of the tangent space of any orbit closure is complex symplectic. Since the tangent space of the orbit closure of $(X, \omega)$ contains a two dimensional complex isotropic subspace, it must have complex dimension at least four and hence coincide with the tangent space to $\M$ at $(X, \omega)$. Therefore, $(X, \omega)$ is generic under the action of $\GL_2(\R)$ and so $\M$ is a branched covering construction of $\mathcal{H}(2)$ as desired.
 \end{proof}
 
 \begin{rem}
 Remember the standing assumption that we are really working with 
 \end{rem}

%
%

\section{The Flat Geometry of Translation Surfaces in Higher Rank Affine Invariant Submanifolds}\label{EDOC}

In this section, we begin the inductive argument. Let $\M$ be a higher rank affine invariant submanifold in a hyperelliptic component of a stratum of abelian differentials. Assume throughout this section that any higher rank affine invariant submanifold in a hyperelliptic components that has dimension strictly smaller than $\M$ is a branched covering construction of a hyperelliptic component of a stratum. Under these hypotheses, the main theorem of the section is the following:

\begin{ssthm}\label{T:Flat-Geometry1}
If $(X, \omega)$ is a horizontally periodic translation surface in $\M$ with twist space and cylinder preserving space coinciding then the following hold:
\begin{enumerate}
\item Any two $\M$-equivalent horizontal cylinders in $(X, \omega)$ have identical heights.
\item If $\M$ is even complex-dimensional then there is exactly one self-adjacent equivalence class of horizontal cylinders.
\item If $s$ is a saddle connection on the boundary of two equivalent horizontal cylinders, then there is a cylinder that contains $s$, intersects it exactly once, and is contained in the equivalence class of horizontal cylinders. 
\end{enumerate}
\end{ssthm}

\noindent By Theorem~\ref{T:Odd1} this result holds for odd-dimensional $\M$. Therefore, throughout this section we will suppose that $\M$ is even-dimensional. Since we have already established the main theorem for rank two affine invariant submanifolds that are even-dimensional, we will suppose furthermore that the rank $r$ of $\M$ is at least three. Suppose finally that $(X, \omega)$ is a horizontally periodic translation surface with twist space and cylinder preserving space coinciding. 

We begin by finding an equivalence class of cylinders that we will use to degenerate $(X, \omega)$ to the boundary of $\M$. 

\begin{sslem}[Leaf Lemma]\label{L:leaf}
There is a horizontal cylinder that is only adjacent to one inequivalent cylinder and perhaps also itself.
\end{sslem}
\begin{proof}
Let $\Gamma$ be the Lindsey half-tree of $(X, \omega)$. Enumerate the equivalence classes of horizontal cylinders $\{1, \hdots, m\}$ and color the vertices of the tree by the corresponding equivalence class. Let $\Gamma'$ be the quotient of $\Gamma$ where each monochromatic connected subtree is collapsed to a single point (colored with the same color as the subtree). Let $\lambda$ be a leaf of the quotient graph $\Gamma'$, i.e. a vertex that connects to at most one other vertex in $\Gamma'$. Let $T$ be the monochromatic connected subtree corresponding to $\lambda$. Since $\lambda$ was a leaf in $\Gamma'$ there is a single vertex $w$ in $T$ so that $T$ is connected to the rest of $\Gamma$ by an edge joining $w$ to a vertex $v$ of a different color. 

Suppose to a contradiction that $T$ contains more vertices than just $w$. Let $C_w$ and $C_v$ be the cylinders in $(X, \omega)$ corresponding to $w$ and $v$ respectively. Let $C$ be the cylinder in $T$ that is different from $C_w$. By the standard position lemma it is possible to shear the equivalence classes containing $C_w$ and $C_v$ so that there is a vertical cylinder $V$ that is contained in $C_w \cup C_v$ and that contains the two saddle connections $s$ and $s'$ that connect $C_w$ to $C_v$. By the cylinder proportion theorem, there is a cylinder $V'$ that is equivalent to $V$ and that passes through $C$. Since the monochromatic tree $T$ is connected to the rest of $\Gamma$ through $w$ and since $V$ contains the saddle connections connecting $C_v$ to $C_w$ it follows that $V'$ must be contained in $T$ (since it cannot escape into the rest of $\Gamma$ through $s$ and $s'$). By the cylinder proportion theorem, since $V'$ is contained entirely in one equivalence class so is $V$. This contradicts the assumption that $V$ intersects both $C_w$ and $C_v$, which are inequivalent cylinders. Therefore, $T$ contains a single vertex $w$ and $w$ is only adjacent to one inequivalent cylinder $v$ and perhaps also itself. 
 \end{proof}
 
We now fix notation that we will use for the remainder of the section. Let $L$ be a horizontal leaf cylinder on $(X, \omega)$ and let $\cC_0$ be the equivalence class of horizontal cylinders that it belongs to. Suppose that the only distinct cylinder that $L$ is adjacent to is $L'$, which belongs to the equivalence class $\cC_1$. Suppose furthermore, using the standard position lemma, that $L$ and $L'$ are in standard position and that $W$ is the resulting vertical cylinder that passes between them. Let $\cW$ be the equivalence class of vertical cylinders that contains $W$.

By the vertical collapse lemma (Lemma~\ref{collapse0}), collapsing $\W$ results in a translation surface $(Y, \eta)$ that is a disjoint union of translation surfaces in hyperelliptic components of strata. Since collapsing $\W$ is a path in $\M$, the surface $(Y, \eta)$ belongs to the Mirzakhani-Wright partial compactification of $\M$. 

\begin{sslem}\label{L:N-dimension}
If $(Z, \zeta)$ is a component of $(Y, \eta)$ that contains a cylinder from $\cC_1$ and $\N$ is the affine invariant submanifold in the boundary of $\M$ that contains $(Z, \zeta)$ then the following hold:
\begin{enumerate}
\item $\N$ is either $2r-3$ or $2r-2$ complex-dimensional 
\item The twist space and cylinder preserving space on $(Z, \zeta)$ coincide. 
\end{enumerate}
\end{sslem}
\begin{proof}
Suppose that $(Z, \zeta)$ is a component of $(Y, \eta)$ that contains a cylinder that belonged to $\cC_1$. Suppose that $\cC$ is any equivalence class of horizontal cylinders excluding $\cC_0$. We will show that $(Z, \zeta)$ contains a cylinder belonging to $\cC$ by inducting on the distance $d$ from $\cC$ to $\cC_1$ in the Lindsey tree. The $d = 0$ base case, i.e. that $(Z, \zeta)$ contains a cylinder that belonged to $\cC_1$ holds by assumption. 

Now suppose that $\cC$ is distance $d > 0$ away from $\cC_1$ in the Lindsey tree and that a cylinder from any equivalence class that is distance less than $d$ away from $\cC$ appears on each component of $(Y, \eta)$. Let $C$ be a cylinder in $\cC$ that is adjacent to a cylinder $D$ that is distance $d-1$ away from $\cC_1$ in the Lindsey tree. Let $\cC'$ be the equivalence class of cylinders containing $D$. The standard position lemma implies that we may apply a standard shear to $\cC$ to put $C$ and $D$ in transverse standard position. In particular, there is a cylinder $V$ that is contained in $C \cup D$ and intersects the saddle connection $s'$ on their boundary exactly once. Let $\V$ be the equivalence class of cylinders containing $V$. By the cylinder proportion theorem, any cylinder in $\V$ must be contained exclusively in cylinders contained in $\cC$ or $\cC'$. In particular, the cylinders do not pass through saddle connections connecting $\cC_0$ to $\cC_1$, which are precisely the ones altered in the collapse. Therefore, all of the cylinders in $\V$ persist on $(Y, \eta)$. By the cylinder proportion theorem, every cylinder in $\cC$ and $\cC'$ are intersected by a cylinder in $\V$. Since there is a cylinder in $\cC'$ on $(Z, \zeta)$ by the induction hypothesis, it follows that there is a cylinder in $\V$ on $(Z, \zeta)$ as well. Therefore, there is a cylinder equivalent to $\cC$ on $(Z, \zeta)$ as desired.

The degeneration theorem of Mirzakhani-Wright~\cite{MirWri} (Theorem~\ref{MirWri} in this paper, Theorem 2.7 in \cite{MirWri}) implies that the rank of $\N$ is strictly less than $r$. Since a cylinder from all but one horizontal equivalence class persists on $(Z, \zeta)$ it follows from the twist space degeneration lemma (Lemma~\ref{L:tsdl} part 1) that the dimension of the twist space of $(Z, \zeta)$ is $r-1$. The dimension of the twist space is at most $\mathrm{rank}(\N) + \mathrm{rel}(\N)$. Since $\N$ is contained in a hyperelliptic component, $\mathrm{rel}(\N) \leq 1$ and so $\N$ is either rank $r-1$ or rank $r-2$ and rel $1$. Since the complex-dimension of an affine invariant submanifold $\N$ is $2\cdot \mathrm{rank}(\N) + \mathrm{rel}(\N)$ we see that $\N$ has complex dimension at least $2r-3$. 

We will now argue that $\N$ has dimension at most $2r-2$. Let $C$ be a cylinder in $\cC_1$ that belongs to $(Z, \zeta)$. Let $(Y_1, \eta_1)$ be the surface formed by collapsing $\cC_0$. While we cannot apply the horizontal collapse lemma to $(Y_1, \eta_1)$ - since we have not verified that $\cC_0$ is not self-adjacent - we can still apply the degeneration theorem. Let $(Z_1, \eta_1)$ be the component of $\cC_0$ that contains $C$ and let $\N_1$ be the affine invariant submanifold in the boundary of $\M$ that contains $(Z_1, \zeta_1)$. Since the cylinders in $\W$ continue to persist on $(Z_1, \zeta_1)$ we may collapse them to pass to a translation surface $(Y_2, \eta_2)$. However, $(Y_2, \eta_2)$ could have been produced in a single degeneration by degenerating $\W$. In particular, $(Y_2, \eta_2)$ is a union of components of $(Y, \eta)$. Since the cylinder $C$ persists on $(Y_2, \eta_2)$ one of those components that is contained in $(Y_2, \eta_2)$ is $(Z, \zeta)$. 

This implies that $\N_1$ is in the boundary of $\N$ and that $\N$ is in the boundary of $\N_1$.  By the degeneration theorem of Mirzakhani-Wright~\cite{MirWri}, the boundary of an affine invariant submanifold has dimension that is strictly smaller than that of the affine invariant submanifold. This implies that $\N$ has dimension at most $2r-2$.  Since the dimension of the twist space is maximal dimensional on $(Z, \zeta)$ it follows that the twist space and cylinder preserving space coincide. 
\end{proof}

The upper bound on the size of the dimension of $\N$ in the proof Lemma~\ref{L:N-dimension} came from realizing the degeneration as two successive degenerations. In the following, lemma we will exploit this idea again. We will show that if $\cC_0$ contains and is distinct from an equivalence class of cylinders $\cV$, then we can realize the degeneration from $(X, \omega)$ to $(Y, \eta)$ as three successive degenerations. This will force the dimension of any component of $(Y ,\eta)$ to be at most $2r-3$ complex-dimensional. 

\begin{sslem}\label{L:baby-equivalence}
Suppose that there is an equivalence class of cylinders $\cV$ that is contained in and distinct from $\cC_0$, then if $(Z, \zeta)$ is any component of $(Y, \eta)$ that contains a cylinder from $\cC_1$ it has orbit closure of dimension $2r-3$. In particular, $\cC_0$ is the only self-adjacent equivalence class of horizontal cylinders on $(X, \omega)$.
\end{sslem}
\begin{proof}
Fix a component $(Z, \zeta)$ of $(Y, \eta)$ that contains a cylinder $C$ that was previously part of $\cC_1$ on $(X, \omega)$. We arrived at $(Z, \zeta)$ from a single degeneration, namely degenerating $\W$, but now we will show that we can also arrive at $(Z, \zeta)$ through three successive degenerations. 

Let $(Y_1, \eta_1)$ be the translation surface that results from collapsing $\cV$. By the vertical collapse lemma, $(Y_1, \eta_1)$ belongs to a disjoint union of translation surfaces in hyperelliptic components of strata. Let $(Z_1, \zeta_1)$ be that component that contains the cylinder $C$ that was previously in $\cC_1$. Let $\N_1$ be the affine invariant submanifold in the boundary of $\M$ that contains $(Z_1, \zeta_1)$. 

Since cylinders from $\cC_0$ appear on $(Z_1, \zeta_1)$ we may collapse them. Let $(Y_2, \eta_2)$ be the resulting translation surface. Let $(Z_2, \zeta_2)$ be the component of $(Y_2, \eta_2)$ that contains $C$ and let $\N_2$ be the affine invariant submanifold containing $(Z_2, \zeta_2)$ that is contained in the boundary of $\N_1$. 

Finally, cylinders from $\W$ persist on $(Z_2, \zeta_2)$ and we may collapse them. Let $(Y_3, \eta_3)$ be the resulting translator surface; let $(Z_3, \zeta_3)$ be a component of $(Y_3, \eta_3)$ containing $C$ and let $\N_3$ be the orbit closure of $(Z_3, \zeta_3)$. 

Observe that the three successive degenerations that we constructed to produce $(Z_3, \zeta_3)$ produce the same translation surface as produced if $\W$ is collapsed. Since both $(Z_3, \zeta_3)$ and $(Z, \zeta)$ contain the image of the cylinder $C$, it follows that they are the same translation surface. By the degeneration theorem of Mirzakhani-Wright~\cite{MirWri} (Theorem~\ref{MirWri} in this paper, Theorem 2.7 in \cite{MirWri}) the complex dimension of $\N_i$ is at most $2r-i$. Since the the orbit closure $\N$ of $(Z, \zeta)$ has complex-dimension at least three and since $\N$ corresponds to $\N_3$, which has complex dimension at most $2r-3$ we see that the $\N$ is exactly $2r-3$ complex-dimensional. 

Every component of $(Y, \eta)$ either contains only cylinders that belonged to $\cC_0$ or it contains a cylinder that belonged to $\cC_1$ and has orbit closure of dimension $2r-3$ with twist space and cylinder preserving space coinciding. In the latter case, Theorem~\ref{T:Odd1} and Theorem~\ref{T:Odd2} imply that no two equivalent cylinders on $(X, \omega)$ remain adjacent on a component of $(Y, \eta)$ that contains a cylinder from $\cC_1$. This implies that only cylinders in $\cC_0$ were self-adjacent. 
\end{proof}

Now we are almost in a position to show that $(X, \omega)$ satisfies the final two properties in Theorem~\ref{T:Flat-Geometry1}. The following lemma almost proves two out of three of the properties we must show to prove the main theorem of this section. The ``almost" is needed because we are about to show that $(X, \omega)$ has at most one self-adjacent equivalence class, whereas Theorem~\ref{T:Flat-Geometry1} states that it has exactly one. The final result of the section will be circling around and improving ``at most one" to ``exactly one".

\begin{sslem}~\label{L:half-of-T}
There is at most one horizontal equivalence class on $(X, \omega)$ that is self-adjacent and given any saddle connection $s$ connecting two equivalent horizontal cylinders there is a cylinder $V$ contained in the equivalence class, containing $s$, and intersecting $s$ exactly once.
\end{sslem}
\begin{proof}
Suppose first that $\cC$ is an equivalence class that is not equal to $\cC_0$ or $\cC_1$ that is self-adjacent. Let $s$ be any saddle connection on the boundary of two cylinders in $\cC$. When the equivalence class $\W$ is collapsed to form $(Y, \eta)$ the horizontal cylinders in $\cC$ remain unaltered. Let $(Z, \zeta)$ be the component of $(Y, \eta)$ on which $s$ persists and let $\N$ be the affine invariant submanifold in the boundary of $\M$ that contains $(Z, \zeta)$. By Lemma~\ref{L:N-dimension}, the twist space and cylinder preserving space of $(Z, \zeta)$ coincide. It follows from Theorem~\ref{T:Odd1} and Theorem~\ref{T:Odd2} that $\N$ cannot have odd dimension since it contains two cylinders from $\cC$ that are adjacent. Therefore, $\N$ is higher rank. By the induction hypothesis, there is a cylinder $V$ that is contained in $\cC$, that contains $s$, and that intersects $s$ exactly once.  

Since the collapse of $\W$ fixes the cylinders in $\cC$, it follows that on $(X, \omega)$ the cylinder $V$ persists. Let $\V$ be the cylinders that are equivalent to $V$ on $(X, \omega)$. By the cylinder proportion theorem, each cylinder in $\V$ is contained in the union of cylinders in $\cC$ and each cylinder in $\cC$ intersects a cylinder in $\V$. The components of $(Y, \eta)$ either consist entirely of cylinders from $\cC_0$ or they contain every equivalence class of horizontal cylinder except for $\cC_0$. In the latter case, the component contains a cylinder from $\cC$ and hence a cylinder from $\cV$. This implies that on these components $\cC$ remains self-adjacent. By the induction hypothesis, there is exactly one equivalence class of adjacent cylinders on any such component of $(Y, \eta)$ and so we have shown that the only equivalence classes of cylinders on $(X, \omega)$ that were self-adjacent are $\cC$ and possibly $\cC_0$.


Suppose now to a contradiction, that both $\cC$ and $\cC_0$ are self-adjacent on $(X, \omega)$. Let $(Y_1, \eta_1)$ be the translation surface that results from collapsing $\cV$. By the vertical collapse lemma, $(Y_1, \eta_1)$ is a disjoint union of translation surface in hyperelliptic components of strata. Let $(Z_1, \zeta_1)$ be any component of $(Y_1, \eta_1)$ that contains two adjacent cylinders in $\cC_0$. This component necessarily contains a cylinder from $\cC_0$ and $\cC_1$ and hence its twist space is at least two-dimensional. However, since two cylinders from $\cC_0$ remain adjacent, Theorem~\ref{T:Odd2} implies that the orbit closure cannot be three complex-dimensional. Therefore, the orbit closure of $(Z_1, \zeta_1)$ is higher rank. By the induction hypothesis, it follows that there is a cylinder $V'$ that passes through the saddle connection connecting the two cylinders in $\cC_0$. However, by Lemma~\ref{L:baby-equivalence} this contradicts the claim that $\cC$ is self-adjacent and we have a contradiction. 

We have shown that if $\cC$ is an equivalence class of cylinders that does not coincide with $\cC_0$ or $\cC_1$ and that is self-adjacent, then the conclusion of the lemma holds. Therefore, it suffices to consider the case where the only two equivalence classes of horizontal cylinders that are possibly self-adjacent are $\cC_0$ and $\cC_1$. The arguments that we are about to make are very similar to the preceding arguments and can be skipped by readers who are only looking for the thrust of the argument.

Let $\cC$ be any other equivalence class of cylinders. By assumption, $\cC$ is not self-adjacent. Collapse it and let $(Y_2, \eta_2)$ be the resulting translation surface. By the horizontal collapse lemma (Lemma~\ref{collapse1}), the translation surface $(Y_2, \eta_2)$ is a disjoint union of translation surfaces in hyperelliptic components of strata of abelian differential.  

Suppose first that $\cC_0$ is self-adjacent on $(X, \omega)$ and let $s$ be any saddle connection on the boundary of two cylinders in $\cC_0$. Let $(Z_2, \zeta_2)$ be the component of $(Y_2, \eta_2)$ that contains the saddle connection $s$. Let $\N_2$ be the affine invariant submanifold in the boundary of $\M$ that contains $(Z_2, \zeta_2)$. Since $(Z_2, \zeta_2)$ contains a horizontal cylinder from $\cC_0$ and a horizontal cylinder from $\cC_1$, the twist space is at least two dimensional. The dimension of $\N_2$ cannot be three since then Theorem~\ref{T:Odd2} would imply that no two cylinders in $\cC_0$ could be adjacent. Therefore, $\N_2$ is higher rank and so the induction hypothesis implies that there is a cylinder $V$ that is contained in $\cC_0$, contains $s$, and intersects $s$ exactly once. Lemma~\ref{L:baby-equivalence} implies that $\cC_1$ cannot be self-adjacent on $(X, \omega)$ and so $\cC_0$ is the unique self-adjacent equivalence class.

The final case to consider is the case where only $\cC_1$ is self-adjacent. Let $s$ be a saddle connection on the boundary of two cylinders in $\cC_1$. Shrink the cylinders in $\W$ so that no cylinder has length longer than the saddle connection $s$. Collapse $\cC_0$ to form a translation surface $(Y_3, \eta_3)$. Since we have supposed that $\cC_0$ is not self-adjacent, the horizontal collapse lemma (Lemma~\ref{collapse0}) implies that $(Y_3, \eta_3)$ is a disjoint union of translation surfaces in hyperelliptic components of strata. Let $(Z_3, \zeta_3)$ be the component of $(Y_3, \eta_3)$ that contains $s$. By the induction hypothesis there is a cylinder $V$ that is entirely contained in $\cC_1$, that contains $s$, and that passes through $s$ exactly once. Moreover, since all the horizontal saddle connections in $\W$ are smaller than $s$, it follows that $V$ passes through no horizontal saddle connection contained in $\W$. In particular, $V$ persists as a cylinder on $(X, \omega)$ as desired. 
\end{proof}

It remains for us to show that equivalent horizontal cylinders have identical heights on $(X, \omega)$. The strategy for establishing this claim is revealed in the following result. 

 \begin{sslem}
 Let $\cC_1$ and $\cC_2$ be two adjacent equivalence classes of horizontal cylinders on $(X, \omega)$ and suppose that at least one of them is not self-adjacent. If all cylinders in $\cC_1$ have identical heights, then all cylinders in $\cC_2$ do as well.
 \end{sslem}
 \begin{proof}
 If $\cC_1$ is self-adjacent then let $C_2$ be the tallest cylinder in $\cC_2$ that is adjacent to a cylinder in $\cC_1$. Otherwise, let $C_2$ be the smallest cylinder in $\cC_2$ that is adjacent to a cylinder in $\cC_1$. Let $h_2$ be the height of $C_2$. Since all the cylinders in $\cC_1$ have identical height by assumption, call that height $h_1$. Let $C_1$ be a cylinder in $\cC_1$ that borders $C_2$. 
 
By the standard position lemma, assume that $C_1$ and $C_2$ are in standard position and that $V$ is the vertical cylinder that passes through them. Let $V'$ be any other vertical cylinder equivalent to $V$. Let $n_i$ be the number of times that $V'$ intersects a cylinder in $\cC_i$. If $\cC_1$ is self-adjacent then every time that $V'$ enters a cylinder in $\cC_2$ it must exit the cylinder into a cylinder in $\cC_1$. This implies that $n_2 \leq n_1$. If $\cC_1$ is not self-adjacent then the same reasoning implies that $n_1 \leq n_2$. Let $P'$ be the proportion of the area of $V'$ contained in $\cC_2$. 

If $\cC_1$ is self-adjacent, then $C_2$ was assumed to be the tallest cylinder in $\cC_2$ bordering a cylinder in $\cC_1$ and so
\[ P \leq \frac{n_1 h_2}{ n_1 h_1 + n_1 h_2} = \frac{h_2}{h_1 + h_2} \]
\noindent Otherwise, $C_2$ was assumed to be the smallest cylinder in $\cC_2$ bordering a cylinder in $\cC_1$ and so
\[ P \geq \frac{n_1 h_2}{ n_1 h_1 + n_1 h_2} = \frac{h_2}{h_1 + h_2}   \]
 \noindent By the cylinder proportion theorem, the percent of time that $V'$ spends in $\cC_2$ is identical to the percent of time that $V$ spends in $\cC_2$, which is $\frac{h_2}{h_1+h_2}$. Therefore, $n_1 = n_2$ and the height of every cylinder in $\cC_2$ that $V'$ passes through is $h_2$. By the cylinder proportion theorem, for every cylinder in $\cC_2$ there is a cylinder equivalent to $V$ that passes through it and so all cylinders in $\cC_2$ have height $h_2$ as desired.
 \end{proof}
 
 We will use the lemma predominantly in the following rephrased form:
 
 \begin{sslem}[Identical Heights Lemma]\label{L:IHL}
 If $(X, \omega)$ has exactly one self-adjacent equivalence class of horizontal cylinders, then if any equivalence class of horizontal cylinders has the property that all cylinders in it have identical heights, then all equivalence classes of horizontal cylinders have this property.
 \end{sslem}

We will use the identical heights lemma to establish the first claim of the main theorem of the section (Theorem~\ref{T:Flat-Geometry1}). 

\begin{sslem}
Any two equivalent horizontal cylinders on $(X, \omega)$ have identical heights.
\end{sslem}
\begin{proof}
By the identical heights lemma (Lemma~\ref{L:IHL}) it suffices to show that any two cylinders in $\cC_1$ have identical heights. Let $\cC$ be any equivalence class of cylinders apart from $\cC_0$ and $\cC_1$. If $\cC$ is not self-adjacent then let $(Y_1, \eta_1)$ be the translation surface formed by collapsing $\cC$. If $\cC$ is self-adjacent, then it contains an equivalence class $\V$ of transverse cylinders and let $(Y_1, \eta_1)$ be the translation surface formed by collapsing $\V$.

In the case where $\cC$ is self-adjacent, the vertical collapse lemma (Lemma~\ref{collapse0}) implies that $(Y_1, \eta_1)$ is a disjoint union of translation surfaces in hyperelliptic components of strata. Any component $(Z_1, \eta_1)$ of $(Y_1, \eta_1)$ that contains a cylinder from $\cC_1$ must also contain a cylinder from $\cC_0$ and hence have twist space of dimension at least two. Let $\N_1$ be the orbit closure of $(Z_1, \zeta_1)$. If $\N_1$ is higher rank, then the induction hypothesis implies that any two cylinders in $\cC_1$ that persist on $(Z_1, \zeta_1)$ have identical heights. Otherwise, $\N_1$ is three dimensional and Theorem~\ref{T:Odd2} implies that any two cylinders in $\cC_1$ that persist on $(Z_1, \zeta_1)$ have identical heights.

In the case where $\cC$ is not self-adjacent, the horizontal collapse lemma (Lemma~\ref{collapse1}) implies that $(Y_1, \eta_1)$ is a disjoint union of translation surfaces in hyperelliptic components of strata. Any component $(Z_1, \zeta_1)$ of $(Y_1, \eta_1)$ that contains a cylinder from $\cC_1$ contains a cylinder from $\cC_0$. The preceding argument implies that any cylinder in $\cC_1$ that persist on the same component of $(Y_1, \eta_1)$ have identical  heights. 

Let us rephrase this observation in the language of Lindsey trees. Let $\Gamma$ be the Lindsey tree of $(X, \omega)$. Let $T$ be any connected subtree of $\Gamma$ that only consists of vertices corresponding to cylinders in $\cC_0$ or $\cC_1$. We have shown that any two cylinders in $\cC_1$ that belong to $T$ have identical heights. Suppose that $T_1$ and $T_2$ are connected subtrees of $\Gamma$ that only consist of vertices in $\cC_0$ and $\cC_1$. Suppose too that $T_1$ and $T_2$ are maximal in the sense that they cannot be enlarged by adding another vertex in $\cC_0$ or $\cC_1$. We will say that $T_1$ and $T_2$ are adjacent if there is a path between them that does not pass though a vertex in $\cC_0$ or $\cC_1$. It suffices to show that if $T_1$ and $T_2$ are adjacent then all cylinders in $\cC_1$ in $T_1$ and $T_2$ have identical heights. For any two such trees, there are cylinders in $\cC_1$ - say $C_1$ on $T_1$ and $C_2$ on $T_2$ - that belong to the same component of $(Y, \eta)$  - the translation surface that we formed by collapsing the cylinders $\W$ at the beginning of the section. 

Notice that we already have shown that if two cylinders from $\cC_1$ belong to the same component of $(Y, \eta)$ then they have identical heights. This follows because every component of $(Y, \eta)$ that has a cylinder from $\cC_1$ is a translation surface with twist space and cylinder preserving space coinciding and with an orbit closure of dimension at least $2r-3$. If the orbit closure is exactly three dimensional, then Theorem~\ref{T:Odd2} implies that any two cylinders from $\cC_1$ appearing together on the component have identical heights. Otherwise, the orbit closure of the component is higher rank and the any two cylinders from $\cC_1$ appearing on the component have identical heights by the induction hypothesis. 
\end{proof}

Now we have almost completed the proof of the main theorem of the section (Theorem~\ref{T:Flat-Geometry1}), but as discussed earlier there is one missing ingredient. So far we have shown that on $(X, \omega)$ there is at most one self-adjacent equivalence class. We must show that there is exactly one. 

\begin{sslem}\label{adjacency}
$(X, \omega)$ contains a self-adjacent $\M$-equivalence class of horizontal cylinders. 
\end{sslem}
\begin{proof}
Suppose not to a contradiction. Since, a fortiori, no cylinder can be adjacent to itself the Lindsey tree $\Gamma$ of $(X, \omega)$ has no half-edges and hence is a tree (and not just a half-tree). By Theorem~\ref{reldef}, if $V(\Gamma)$ is the collection of horizontal cylinders on $(X, \omega)$ then 
\[ \eta = \sum_{c \in V(\Gamma)} (-1)^{d(c, c_0)} \gamma_c^* \]
is a relative deformation where $c_0$ is a fixed horizontal cylinder and $d(c,c_0)$ is the distance between cylinders $c$ and $c_0$ in the Lindsey tree. Since $\M$ is an even dimensional affine invariant submanifold in $\hyp(g-1,g-1)$ its tangent space contains no relative deformation. Therefore, we can produce a contradiction if we can show that $\eta$ belongs to the tangent space of $\M$ at $(X, \omega)$. 

First we formulate the problem as a graph theory problem using trees. In the Lindsey tree the vertices correspond to horizontal cylinders. Since cylinders divide into equivalence classes, if there are $m$ equivalence classes of horizontal cylinders we imagine that the Lindsey tree is colored using colors $\{1, \hdots, m\}$ so that a vertex is colored by the equivalence class it belongs to. The assumption that no equivalent cylinders are adjacent means that no two vertices of the same color are adjacent. We will find one more constraint on the Lindsey tree using the cylinder proportion theorem. By the standard position lemma, given two adjacent cylinders $C$ and $D$ we can shear their equivalence classes to put them in standard position and in particular find a cylinder $V$ that intersects $C$ and $D$ and that is contained in $C \cup D$. Let $\cC$ be the equivalence class containing $C$ and $\mathcal{D}$ the one containing $D$. By the cylinder proportion theorem, if $V'$ is a cylinder equivalent to $V$ then $V'$ must intersect at least one cylinder in $\cC$ and at least one cylinder in $\mathcal{D}$ and $V'$ must be contained in the union of cylinders in $\cC$ and $\mathcal{D}$. Since equivalent cylinders are not adjacent, if $V'$ is equivalent to $V$ it must alternate between passing through cylinders in $\cC$ and cylinders in $\mathcal{D}$. The cylinder proportion theorem guarantees that for any cylinder in $\cC$ or $\mathcal{D}$ there is a cylinder $V'$ equivalent to $V$ that intersects the cylinder. In particular, this means that every cylinder in $\cC$ is adjacent to one in $\mathcal{D}$ and vice versa. 

\begin{sublem}
Let $\Gamma$ be a finite tree whose vertices are colored with colors $\{1, \hdots, m\}$. Suppose that
\begin{enumerate}
\item (No self-adjacency) No two vertices of the same color are adjacent.
\item (Cylinder Proportion Theorem)  If $v$ and $w$ are vertices of the same color and $v$ borders a vertex of color $c$ then $w$ does as well. 
\end{enumerate}
Then vertices of the same color are an even distance apart in $\Gamma$.
\end{sublem}
\begin{proof}
Induct on the number of colors $m$. For $m=2$ the result is clear. Suppose now that $m>2$. Let $a$ be a leaf of $\Gamma$ connected to vertex $b$ and suppose that $a$ and $b$ have colors $1$ and $2$ respectively. Since $a$ is a leaf, the only vertex it borders is $b$. The cylinder proportion hypothesis then implies that any vertex of color $1$ can only border a vertex of color $2$. 

Let $v$ and $w$ be two vertices of the same color and let $[v, w]$ be the geodesic between them in $\Gamma$. Let $\Gamma'$ be the colored tree that results from collapsing each maximal connected subtree containing only vertices of color $1$ or $2$ to a point of color $0$. We see that the two properties we assumed about $\Gamma$ - no self-adjacency and the cylinder proportion theorem - descend to $\Gamma'$. Let $[v]$ and $[w]$ be the images of $v$ and $w$ in $\Gamma'$. By the inductive hypothesis, $d([v], [w])$ is even. It suffices to show that $d(v, w)$ is even.

First suppose that $v$ and $w$ are not color 1. Let $Z$ be the collection of points of color $0$ on $\left[ [v], [w] \right]$ and for each $z \in Z$ let $T_z$ be the maximal connected subtree of vertices of color $1$ or $2$ in $\Gamma$ corresponding to $z$. As discussed above, if $v$ and $w$ are not color $1$ then whenever the geodesic $[v, w]$ enters the tree $T_z$ it does so through a vertex of color $2$. Similarly, if the geodesic exits the tree $T_z$ it does so through a vertex of color $2$. It follows that the length $\ell_z$ the geodesic $[v,w]$ travels in the tree $T_z$ is even, since it is a path in $T_z$ between two vertices of color $2$. Since $\ds{ d(v, w) = d([v], [w]) + \sum_{z \in Z} \ell_z }$ is the sum of even numbers it is even.

Now suppose $v$ and $w$ are of color $1$. There are unique vertices of color $2$, call them $v'$ and $w'$, such that $d(v,w) = 2 + d(v', w')$. Since the distance between any two points of color $2$ is even, it follows that the distance between $v$ and $w$ is even as well. 
 \end{proof}
 
Now we will proceed with the proof of Lemma~\ref{adjacency}. Since equivalent horizontal cylinders have identical heights in $(X, \omega)$ by assumption it follows that if $\cC$ is an equivalence class then, up to scaling, the standard shear is $\sigma_{\cC} = \sum_{c \in \cC} \gamma_c^*$. Fix a cylinder $c_0$. By the sublemma, the collection of cylinders that are an even distance away from $c_0$ is a union of equivalence classes $\cC_0, \hdots, \cC_k$ and the collection of cylinders that are an odd distance away from $c_0$ is also a union of equivalence classes $\cC_{k+1}, \hdots, \cC_m$. It follows that the relative deformation $\eta$ can be written as follows: 
 \[ \eta = \sum_{i=1}^k \sigma_{\cC_i} - \sum_{i=k+1}^m \sigma_{\cC_i} \]
\noindent Since each standard shear belongs to the tangent space of $\M$ at $(X, \omega)$ and since $\eta$ is a linear combination of them, it follows that the tangent space also contains the relative deformation $\eta$. This is a contradiction. 
 \end{proof}

%
%

\section{The Isogenous Cylinder Lemma}\label{BCCC}

In Section~\ref{S:dim4} we showed that if $\M$ is an affine invariant submanifold in a hyperelliptic component and $\M$ is four complex-dimensional then $\M$ is a branched covering construction of $\mathcal{H}(2)$. The proof revolved around showing that equivalent cylinders were isogenous. To show that equivalent cylinders were isogenous we needed three results about translation surfaces $(X, \omega)$ with twist space and cylinder preserving space coinciding.
\begin{enumerate}
\item Equivalent horizontal cylinders have identical heights.
\item Equivalent saddle connections have identical lengths.
\item Equivalent saddle connections alternate around the boundary of horizontal cylinders.
\end{enumerate}
\noindent We will now prove these claims for a general affine invariant submanifold $\M$. Throughout this section we will suppose that $\M$ is a higher rank affine invariant submanifold in a hyperelliptic component of a stratum and that $\M$ has complex-dimension at least $5$. Let $(X, \omega)$ be a translation surface in $\M$ that is horizontally periodic and for which the twist space and cylinder preserving space coincide. Suppose furthermore that any higher rank affine invariant submanifold in a hyperelliptic component that has dimension strictly smaller than $\M$ is a branched covering construction. The main result of this section will be that equivalent horizontal cylinders on $(X, \omega)$ are isogenous.

If $S$ is a collection of equivalent saddle connections on $(X, \omega)$ then let $S_t \cdot (X, \omega)$ be the translation surface where all saddle connections in $S$ have been dilated by $t$.

\begin{sslem}\label{good}
If $\M$ is even complex-dimensional then $(X, \omega)$ has exactly one equivalence class of horizontal cylinders, say $\cC$, that is self-adjacent. If $S$ is the collection of horizontal saddle connections connecting two cylinders in $\cC$ then all saddle connections in $S$ have identical lengths and $S_t \cdot (X, \omega)$ belongs to $\M$ for all $t > 0$. 
\end{sslem}
\begin{proof}
By Lemma~\ref{adjacency} there is at least one self-adjacent $\M$-equivalence class $\Cyl$ of horizontal cylinders on $(X, \omega)$. Let $S$ be the collection of horizontal saddle connections that connect two cylinders in $\Cyl$. Suppose that $s$ is the longest saddle connection in $S$ and suppose that it lies on the boundary of cylinders $A$ and $B$ in $\cC$. By Theorem~\ref{T:Flat-Geometry1} there is a cylinder $V$ that is contained in $\cC$, passes through $s$ exactly once, and contains $s$ in its interior. Let $\V$ be the equivalence class of cylinders that contains $V$. Without loss of generality, after perhaps, shearing the surface we may suppose that the cylinders in $\V$ are vertical. Every cylinder in $\V$ has the same height as $V$ and only passes through horizontal saddle connections that lie on the boundary of two cylinders in $\cC$. Since $s$ is the longest saddle connection in $S$, it follows that every cylinder in $\V$ passes exclusively through horizontal saddle connections of length $s$. Let $\cS$ be the collection of horizontal saddle connections through which a cylinder in $\V$ passes. Notice that the standard shear on $\V$ is equivalent to $\cS_t \cdot (X, \omega)$ for some real number $t$

By the vertical collapse lemma (Lemma~\ref{collapse0}), collapsing $\V$ results in a translation surface $(Y, \eta)$ that is a disjoint union of translation surfaces in hyperelliptic components of strata. Since collapsing $\V$ is a path in $\M$, the surface $(Y, \eta)$ belongs to the Mirzakhani-Wright partial compactification of $\M$. 

\begin{sublem}
A cylinder from every $\M$-equivalence class of horizontal cylinders persists on each component of $(Y, \eta)$.
\end{sublem}
\begin{proof}
Let $\cC'$ be an $\M$-equivalence class of horizontal cylinders on $(X, \omega)$. We wish to show that a cylinder from $\cC'$ persists on every component of $(Y, \eta)$. Assume without loss of generality that we have applied $\cS_t$ so that every saddle connection in $\cS$ is shorter than every other horizontal saddle connection on $(X, \omega)$. Proceed by induction on the distance $d$ from $\cC$ to $\cC'$ in the Lindsey tree. The base case, $d = 0$, is immediate since a cylinder from $\cC$ necessarily appears on each component of $(Y, \eta)$. 

Now suppose that $\cC'$ is distance $d > 0$ away from $\cC$ in the Lindsey tree and that a cylinder from any equivalence class that is distance less than $d$ away from $\cC$ appears on each component of $(Y, \eta)$. Let $C$ be a cylinder in $\cC'$ that is adjacent to a cylinder $D$ that belongs to an equivalence class that is distance $d-1$ away from $\cC$ in the Lindsey tree. The standard position lemma implies that we may apply a standard shear to $\cC'$ to put $C$ and $D$ in transverse standard position. In particular, there is a cylinder $W$ that is contained in $C \cup D$ and intersects the saddle connection $s'$ on their boundary exactly once. The horizontal distance across $W$ is the length of $s'$ and every cylinder equivalent to $W$ is the same horizontal distance across. In particular, this means that $W$ cannot intersect any horizontal saddle connection in $\cS$ since all of these are strictly shorter than $s'$. Let $\mathcal{W}$ be the collection of cylinders equivalent to $W$. Since no cylinder in $\mathcal{W}$ passes through a saddle connection in $\cS$, all of the cylinders in $\mathcal{W}$ persist on $(Y, \eta)$. By the cylinder proportion theorem, every cylinder in $\cC'$ and every cylinder equivalent to $D$ is intersected by a cylinder in $\mathcal{W}$. Since there is a cylinder equivalent to $D$ on every component of $(Y, \eta)$ by the induction hypothesis, it follows that there is a cylinder equivalent to $\mathcal{W}$ on every component of $(Y, \eta)$. Therefore, there is a cylinder equivalent to $\cC'$ on every component of $(Y, \eta)$ as desired.
 \end{proof}
 
 Let $(Z, \zeta)$ be any component of $(Y, \eta)$ and let $\N$ be the orbit closure of $(Z, \zeta)$. By assumption a cylinder from every equivalence class of horizontal cylinders persists on $(Z, \zeta)$ and the twist space of $(Z,\zeta)$ contains the image of the standard shears of the horizontal equivalence classes of cylinders on $(X, \omega)$ by the degeneration theorem of Mirzakhani-Wright~\cite{MirWri} (Theorem~\ref{MirWri} in this paper, Theorem 2.7 in \cite{MirWri}). In particular, this implies that the twist space on $(Z, \zeta)$ is at least $r$ dimensional. The dimension of the twist space of a horizontally periodic translation surface in an affine invariant submanifold $\N$ is at most dimension $\mathrm{rank}(\N) + \mathrm{rel}(\N)$. Since $\N$ belongs to a hyperelliptic component, $\mathrm{rel}(\N) \leq 1$. By the degeneration theorem of Mirzakhani-Wright, $\mathrm{rank}(\N) < r$. These observations imply that the twist space is exactly $r$ dimensional and that $\mathrm{rank}(\N) = r-1$ and $\mathrm{rel}(\N) = 1$. 

Since the twist space has maximal dimension, the twist space and cylinder preserving space on $(Z, \zeta)$ coincide. By the twist space degeneration lemma (Lemma~\ref{L:tsdl} part 3), if $\N$ is rank one then $\M$ equivalent horizontal cylinders are not adjacent on $(Z, \zeta)$. If $\N$ is higher rank, then two horizontal cylinders that persist on $(Z, \zeta)$ are $\N$-equivalent if and only if they were $\M$-equivalent on $(X, \omega)$. Theorem~\ref{T:Odd1} then implies that when $\N$ is higher rank, $\M$-equivalent cylinders are not adjacent on $(Z, \zeta)$. Since these conclusions hold for all components of $(Y, \eta)$ it follows that once $\V$ is collapsed that no $\M$-equivalent cylinders remain adjacent.
 \end{proof}

One of the main ingredients in showing that rank two rel zero affine invariant submanifolds were branched covers of $\mathcal{H}(2)$ was to show that on horizontally periodic translation surfaces with twist space and cylinder preserving space coinciding, equivalent saddle connections had identical lengths. This was established in Lemma~\ref{L:bcch2}. We will now establish the same result in the present more general setting.

\begin{sslem}[Saddle Connection Dilation Lemma]\label{SCDL}
Let $S$ be an equivalence class of horizontal saddle connections on $(X, \omega)$, then every element of $S$ has the same length and $S_t \cdot (X, \omega) \in \M$ for all $t$.
\end{sslem}
\begin{proof}
Let $\Cyl_0$ and $\Cyl_1$ be two distinct adjacent $\M$-equivalence classes of horizontal cylinders on $(X, \omega)$. Let $S$ be the equivalence class of horizontal saddle connections connecting them. Let $s \in S$ have length $\ell$ and be the longest element of $S$. By Lemma~\ref{good} suppose without loss of generality that any saddle connection connecting two $\M$-equivalent cylinders has length strictly smaller than $\ell$. Suppose that $s$ lies on the boundary of $C_0 \in \Cyl_0$ and $C_1 \in \Cyl_1$. By the standard position lemma suppose without loss of generality that $C_0$ and $C_1$ are in standard position and let $V$ be the resulting cylinder. Let $\V$ be the $\M$-equivalence class of vertical cylinders that contains $V$. Since every cylinder in $\V$ has width $\ell$ it follows that no cylinder in $\V$ passes through a horizontal saddle connection that connects two $\M$-equivalent cylinders. 

\begin{sublem}
$\V$ contains every element of $S$.
\end{sublem}
\begin{proof}
Suppose to a contradiction that $C$ is a cylinder in $\Cyl_0$ and $D$ is an adjacent cylinder in $\Cyl_1$ so that $\V$ does not contain the saddle connection connecting them. Collapse $\V$ and let $(Y, \eta)$ be the resulting boundary translation surface. By the vertical collapse lemma $(Y, \eta)$ is a disjoint union of connected translation surfaces in hyperelliptic components of strata. Let $(Y', \eta')$ be the component of $(Y, \eta)$ containing $C$ and $D$. By assumption, $(Y', \eta')$ contains a cylinder belonging to $\cC_0$ and $\cC_1$. Since collapsing $\V$ only alters saddle connections connecting $\cC_0$ to $\cC_1$ it follows from the cylinder proportion theorem that every equivalence class of horizontal cylinders persists on $(Y', \eta')$. By the twist space degeneration lemma (Lemma~\ref{L:tsdl} part 4) $(Y', \eta')$ cannot contain the image of two adjacent $\M$-equivalent cylinders and $\dim_\C \M = 2r$. By Theorem~\ref{T:Flat-Geometry1} (specifically, by the result shown in Lemma~\ref{adjacency}), $(Y', \eta')$ must contain the image of two adjacent $\M$-equivalent cylinders, which is a contradiction.


 \end{proof}

By hypothesis, every cylinder in $\V$ has the same height. Since $s$ was chosen to be the longest saddle connection in $S$ and since its length is the height $V$, it follows that all saddle connections in $S$ have identical lengths. Since every element of $S$ is contained in $\V$ and since $\V$ only passes through saddle connections in $S$, the cylinder deformation theorem implies that dilating $\V$ horizontally by $t$ for arbitrary $t$ remains in $\M$. Now undoing the shears that put $C_0$ and $C_1$ in standard position remains in $\M$ by the cylinder deformation theorem and is the translation surface $S_t \cdot (X, \omega)$ as desired.
 \end{proof}

To summarize our progress, we have shown that equivalent horizontal cylinders on $(X, \omega)$ have identical heights and that equivalent saddle connections have identical lengths. It remains for us to study the order in which saddle connections appear on the boundary of the horizontal cylinders (and ultimately to use that to show that equivalent cylinders are isogenous). 

First we will need a combinatorial lemma. Before stating it consider the following - the saddle connections on the boundary of a horizontal cylinder are cyclically ordered - suppose that they are labelled $\{1, \hdots, n \}$ from left to right. Each saddle connection belongs to some equivalence class of saddle connections and so we imagine it being colored by the equivalence class that it belongs to. Given two saddle connections $i$ and $j$ we will let $(i,j)$ be the collection of saddle connections running left to right from $i$ to $j$ excluding saddle connections $i$ and $j$. Let $C_{(i, j)}$ be the colors (i.e. equivalence classes) of saddle connections in $(i,j)$ recorded with multiplicity.

\begin{sslem}\label{balls}
Suppose that there is a horizontal cylinder $C$ on $(X, \omega)$ with the following property: if $i \ne j$ and $j \ne k$ are saddle connections on the boundary of $C$ so that
\begin{enumerate}
\item All three saddle connections belong to an equivalence class $c$
\item Neither $C_{(i,j)}$ nor $C_{(j, k)}$ contains the equivalence class $c$
\end{enumerate}
then $C_{(i, j)} = C_{(j, k)}$ as multi-sets. Then the cyclic order of equivalence classes on the boundary of $C$ is periodic, i.e. after renaming the equivalence classes $\{1, \hdots, m \}$  appearing on the boundary of $C$, the equivalence classes of saddle connections on the boundary of $C$ appear in the cyclic order $(1, 2, \hdots, m, 1, 2, \hdots, m, \hdots)$.
\end{sslem}
\begin{proof}
Let $\{1, \hdots, m\}$ be the colors of saddle connections appearing on the boundary of $C$. If every color appears exactly once on the boundary of $C$ then the result holds trivially. Suppose that this is not the case. Let $i$ and $j$ be any two distinct saddle connections of the same color (without loss of generality color $1$) such that $(i,j)$ has the fewest possible elements. Construct an order set of saddle connections of color $1$ as follows. Let $s_1 = i$. Given $s_a$ let $s_{a+1}$ be the first saddle connection to the right of $s_a$ that has color $1$. Stop the process when $s_k = s_1$. It follows that every saddle connection either has color $1$ or belongs to $(s_a, s_{a+1})$ for some $1 \leq a \leq k-1$. Since by assumption $C_{(s_a, s_{a+1})}$ is the same regardless of $a$, it follows that every color $\{2, \hdots, m\}$ appears in $(s_a, s_{a+1})$ for each $a$. Since $(i,j)$ has the fewest number of elements conditional on $i$ and $j$ having the same color, it follows that every color must appear exactly once in $C_{(i,j)}$.

If $(s_1, s_2)$ is empty, then $m = 1$ and the result trivially holds. Therefore, suppose that $(s_1, s_2)$ is nonempty. Suppose after renaming the colors, that the order of colors in $[s_1, s_3]$ is $(1, \hdots, m, 1, c_2, \hdots, c_m, 1)$ where $(c_2, \hdots, c_m)$ is a permutation of $(2, \hdots, m)$. By assumption $(s_1,s_2)$ has the fewest possible elements given that $s_1$ and $s_2$ have the same color. Therefore, $c_m = m$. This in turn implies that $c_{m-1} = m-1$ and so we have that the order of the colors on $[s_1, s_3)$ is $(1, \hdots, m, 1, \hdots, m)$. Iterating this argument gives that the order of the color of the balls is periodic with period $m$.
 \end{proof}

We now use this lemma to show that the cyclic order of equivalence classes of saddle connections appearing on the boundary of a horizontal cylinder $C$ in $(X, \omega)$ is indeed periodic. 

\begin{sslem}[Periodic Ordering Lemma 1]\label{periodic-order}
Let $C$ be a horizontal cylinder on $(X, \omega)$ and suppose that the equivalence classes of saddle connections on the boundary are labelled $\{1, \hdots, m\}$. The cyclic ordering of saddle connections on the boundary of $C$ is (perhaps after relabelling) $(1, 2, \hdots, m, 1, 2, \hdots, m, \hdots)$. 
\end{sslem}
\begin{proof}
By Lemma~\ref{balls} it suffices to show the following: suppose that $i, j, k$ are saddle connections on the top boundary of $C$ in equivalence class $c$. Suppose furthermore that (in the notation of Lemma~\ref{balls}) $(i, j)$ and $(j, k)$ do not contain saddle connections in equivalence class $c$. It suffices to show that $C_{(i, j)}$ and $C_{(k, \ell)}$ coincide as multi-sets.

Let $a$ be any transcendental number. By the saddle connection dilation theorem we may suppose without loss of generality that each saddle connection in equivalence class $k$ has length $a^k$. By using the cylinder deformation theorem we may shear $C$ so that the saddle connection $j$ lies above its image $J(j)$ under the hyperelliptic involution. Suppose moreover that $j$ is contained in a vertical cylinder $V$ that passes through it exactly once and only passes through saddle connections equivalent to $j$. Such a vertical cylinder exists by the standard position lemma in the case that $j$ lies on the boundary of two inequivalent cylinders and by Theorem~\ref{T:Flat-Geometry1} otherwise. 

Let $\V$ be the $\M$-equivalence class of cylinders containing $V$. By the saddle connection dilation lemma, every saddle connection equivalent to $j$ is contained in a cylinder in $\V$ that intersects it exactly once. Moreover, all the cylinders in $\V$ have identical heights. Since the saddle connections of color $c$ appear in the order $(i, j, k)$ on the top boundary of $C$, their images under the hyperelliptic involution appear in the order $(J(k), J(j), J(i))$ on the bottom boundary, where $J$ is the hyperelliptic involution. Since all saddle connections of color $c$ are contained in cylinders in $\V$, it follows that $i$ lies directly vertically above $J(k)$. In particular, this means that the sum of the lengths of the saddle connections in $(i, j)$ coincides with the sum of the length of the saddle connections in $(J(k), J(i))$, which is equal to the sum of the lengths of the saddle connections in $(j, k)$. Since we have assumed that $\M$-equivalence class $k$ of horizontal saddle connections all have length $a^k$ where $a$ is transcendental, it follows that $C_{(i,j)} = C_{(j, k)}$. 
 \end{proof}
 
For the proof of the following lemma it will be useful to introduce the following definition. We will say that an equivalence class of cylinders $\cC$ subsumes an equivalence class of saddle connections $\cS$ on the boundary of two equivalence classes $\cC_1$ and $\cC_2$ if every saddle connection in $\cS$ is contained in a cylinder in $\cC$ that intersects it exactly once and if each cylinder in $\cC$ is contained in collection of cylinders in $\cC_1$ and $\cC_2$. In the case that $\cC_1$ and $\cC_2$ are distinct, the standard position lemma implies that it is always possible to find an equivalence class of cylinders that subsumes $\cS$. Otherwise, Theorem~\ref{T:Flat-Geometry1} implies that this is possible. 

\begin{sslem}[Periodic Ordering Lemma 2]
If $C$ and $D$ are equivalent horizontal cylinders on $(X, \omega)$, then the boundaries of $C$ and $D$ contain the same $\M$-equivalence classes of horizontal saddle connections in the same cyclic order.
\end{sslem}
\begin{proof}
Suppose that $C$ belongs to the equivalence class $\cC_0$. We begin by showing that the boundaries of $C$ and $D$ contain the same equivalence classes of saddle connections. If not, then the saddle connection dilation lemma implies that it is possible to alter the length of one cylinder in $\{C, D \}$ while fixing the other. However, by Wright~\cite{Wcyl} Lemma 4.7, the ratio of the lengths of core curves of equivalent cylinders is constant. This yields a contradiction and therefore, the same equivalence classes of cylinders appear on the boundaries of $C$ and $D$.

Let $\{1, \hdots, m \}$ be the equivalence classes of cylinders appearing on the boundary of $C$ and suppose that they occur in the cyclic order $(1, 2, \hdots, m, 1, 2, \hdots, m, \hdots)$. By the saddle connection dilation lemma we may suppose that each equivalence class of saddle connection appearing on the boundary of $C$ has length $1$. By the cylinder deformation theorem we may suppose that both $C$ and $D$ have height $1$. We will suppose furthermore that if $\cC_0$ is a self-adjacent equivalence class of cylinders, then the equivalence class of saddle connections joining $\cC_0$ to itself is labelled $1$. Let $\T_1$ be a vertical equivalence class of cylinders that subsumes all saddle connection in equivalence class $1$. For every other equivalence class $k$ of saddle connection on the boundary of $C$, there is a saddle connection that connects $C$ to an inequivalent cylinder $C_k$ in equivalence class $\cC_k$. By the standard position lemma, it is possible to shear $\cC_k$ to put $\cC_0$ and $\cC_k$ in transverse standard position. Let $\T_k$ be the resulting collection of cylinders, which pass through $\cC_0$ and $\cC_k$ and so every saddle connection in equivalence class $k$ appears in a cylinder in $\T_k$ and is intersected exactly once by that cylinder. The slope of the core curve of cylinders in $\mathcal{T}_k$ is $\frac{1}{2k}$. By the cylinder proportion theorem $D$ contains a cylinder in $\mathcal{T}_k$ for each $k$ and that $D$ is contained in the union of cylinders in some $\mathcal{T}_k$. The decreasing slopes force the periodic ordering on the boundary of $D$ agrees with the periodic ordering on the boundary of $C$. 
 \end{proof}

\begin{sslem}[Isogeny Lemma]\label{isogeny}
All $\M$-equivalent horizontal cylinders on $(X, \omega)$ are isogeneous. 
\end{sslem}
\begin{proof}
Let $\Cyl$ be an $\M$-equivalence class of horizontal cylinders. Let $C \in \Cyl$ be a cylinder and let $D$ be an adjacent $\M$-inequivalent cylinder. Put $C$ and $D$ into standard position by using the cylinder deformation theorem (suppose that $\Cyl$ is sheared by $\begin{pmatrix} 1 & t \\ 0 & 1 \end{pmatrix}$) - and let $V$ be the resulting vertical cylinder. Let $\V$ be the $\M$-equivalence class of cylinders containing $V$. Label the equivalence class containing the edge connecting $C$ to $D$ by $1$. For any cylinder $v \in \Cyl$ let $\{1, \hdots, m\}$ be the edge equivalence classes attached to $v$ and suppose that they are ordered by $(1, \hdots, m, \hdots)$ around $v$. By assumption every cylinder in $\Cyl$ has height $h$. By the saddle connection dilation lemma every saddle connection in equivalence class $i$ has length $\ell_i$ and $\V$ contains every saddle connection in edge equivalence class $1$. Therefore all cylinders in $\Cyl$ are isogenous to $\begin{pmatrix} 1 & -t \\ 0 & 1 \end{pmatrix}$ applied to the cylinder depicted in Figure~\ref{fig:Isogeny}
\begin{figure}[H]
\centering
\begin{tikzpicture}[scale=.75]
	\draw (0,0) -- (4,0) -- (4,1) -- (0,1) -- (0,0);
	\draw[dotted] (1,0) -- (1,1);
 	\draw[dotted] (2,0) -- (2,1);
	\draw[dotted] (3,0) -- (3,1);
	\node at (.5, -.5) {$1$}; \node at (.5,1.5) {$1$};
	\node at (1.5, -.5) {$m$}; \node at (1.5,1.5) {$2$};
	\node at (2.5, -.5) {$\hdots$}; \node at (2.5, 1.5) {$\hdots$};
	\node at (3.5, -.5) {$2$}; \node at (3.5, 1.5) {$m$};
	\node at (-.5, .5) {$h$};
\end{tikzpicture}
\caption{The cylinder to which all cylinders in $\Cyl$ are isogenous}
\label{fig:Isogeny}
\end{figure}
where labels describe how the top boundary glues to the bottom boundary and where a saddle connection labelled $k$ has length $\ell_k$. 
 \end{proof}

%
%

\section{Higher Rank Affine Invariant Submanifolds are Branched Covering Constructions}\label{S:bcc-criterion}


The goal of this section is to show that all higher rank affine invariant submanifolds in hyperelliptic components of strata are branched covering constructions. We begin by proving an analogous statement at the level of trees. Suppose that $\Gamma$ and $\Gamma'$ are graphs. Define a degree $d$ branched covering of finite graphs to be a simplicial map between graphs $\pi: \Gamma \ra \Gamma'$ such that
\begin{enumerate}
\item $|E_\Gamma| = d \cdot |E_{\Gamma'}|$ where $E_\Gamma$ (resp. $E_{\Gamma'}$) is the edge set of $\Gamma$ (resp. $\Gamma'$).
\item For each vertex $v$ in $\Gamma$ the ramification index $e_v := \frac{ \deg (v) }{\deg \pi(v) }$ is an integer.
\item For each vertex $w$ in $\Gamma'$, $\ds{ \sum_{\pi(v) = w} e_v = d }$.
\end{enumerate}

\begin{sslem}\label{L:graphbcc}
If $f: \Gamma \ra \Gamma'$ is branched covering of finite graphs where $\Gamma$ is a disjoint union of trees and $\Gamma'$ is connected, then $\Gamma'$ is a tree as well. 
\end{sslem}
\begin{proof}
Let $\chi_\Gamma$ denote the Euler characteristic of the graph $\Gamma$. We will first show a Riemann-Hurwitz type formula for branched covers of graphs.
\[ \chi_\Gamma = |V_\Gamma| - |E_\Gamma| = d \cdot |V_{\Gamma'}| - \sum_{v \in V_\Gamma} \left( e_v - 1 \right) - d|E_{\Gamma'}| = d \cdot \chi_{\Gamma'} - \sum_{v \in B} \left( e_v - 1 \right) \]
By assumption $\chi_\Gamma$ is positive. The Riemann-Hurwitz type formula therefore implies that $\chi_{\Gamma'}$ is positive too and therefore must be equal to $1$ since $\Gamma'$ is connected. The connected graphs of Euler characteristic one are exactly trees. 
\end{proof}

\begin{ssthm}
Suppose that $\M$ is a rank $r > 1$ affine invariant submanifold in a hyperelliptic component of a stratum of abelian differentials, then $\M$ is a branched covering construction of $\hyp(2r-2)$ if it is even complex-dimensional and of $\hyp(r-1, r-1)$ if it is odd complex dimensional. The branched covers are branched over zeros of the one-forms and commute with the hyperelliptic involution. 
\end{ssthm}
\begin{proof}
We will proceed by induction on the dimension of $\M$. The base case has been established in Section~\ref{S:dim4}. By Corollary~\ref{useful} it suffices to produce an $\M$-generic horizontally periodic translation surface $(X, \omega)$ with $\Twist \M = \Pres \M$ and so that $(X, \omega)$ is a simple translation covering of a generic surface $(Y, \eta)$ in $\hyp(2r-2)$ or $\hyp(r-1,r-1)$. Enumerate the $\M$-equivalence classes of horizontal cylinders $\{1, \hdots, m\}$ and the equivalence classes of edges $\{1, \hdots, n\}$. Choose two transcendental numbers $a$ and $b$ so that $\Q(a, b)$ is isomorphic as a field to $\Q(x, y)$ where $x$ and $y$ are indeterminates. Using the cylinder deformation theorem and the saddle connection dilation lemma, assume without loss of generality that the height of the cylinders in equivalence class $k$ is $a^k$ and the lengths of saddle connections in equivalence class $k$ is $b^k$.

We will first build the Lindsey tree $\Gamma_Y$ of $(Y, \eta)$. For each $\M$-equivalence class in $(X, \omega)$ add a corresponding node in $(Y, \eta)$. If two distinct equivalence classes are adjacent in $(X, \omega)$, then add an edge connecting the corresponding nodes in $\Gamma_Y$. If an equivalence class is self-adjacent in $(X, \omega)$ then add a half-edge to the corresponding node in $\Gamma_Y$. The cyclic order around each node is specified by the periodic ordering lemmas. This completely specifies $\Gamma_Y$. 

We will now build $(Y, \eta)$ and the simple translation covering $f: (X, \omega) \ra (Y, \eta)$. To each $\M$-equivalence class $\Cyl_i$ of horizontal cylinders let $C_i$ be the shortest cylinder isogenous to the cylinders in $\Cyl_i$ (this will be the one constructed in Lemma~\ref{isogeny}). For each $c \in \Cyl_i$ let $f_c: c \ra C_i$ be the local isometry constructed in Lemma~\ref{isogeny}. Gluing the cylinders $C_i$ together according to $\Gamma_Y$ now constructs a translation surface $(Y, \eta)$. Let $f$ be the map that sends a point $x \in c$ to $f_c(x)$. This map defines a holomorphic degree $d$ covering map $f: X - Z(\omega) \ra Y - Z(\eta)$ so that $f$ pulls $\eta$ back to $\omega$. By the Riemann extension theorem $f$ extends to a holomorphic map $f: X \ra Y$ such that $f^* \eta = \omega$. 

Let $\Gamma'$ be $\Gamma_Y$ with the half-edge deleted (should there be one). Let $\Gamma$ be the Lindsey tree $\Gamma_X$ of $(X, \omega)$ with edges connecting equivalent cylinders deleted. Since $f$ is a holomorphic map between $X - Z(\omega)$ and $Y - Z(\eta)$ it induces a degree $d$ branched covering of graphs between $\Gamma$ and $\Gamma'$.  Lemma~\ref{L:graphbcc} implies that $\Gamma'$ is a tree and hence that $\Gamma_Y$ is a half-tree. 

By Corollary~\ref{LTC}, $(Y, \eta)$ is a translation surface in a hyperelliptic component of a stratum. Moreover, the translation covering $f: (X, \omega) \ra (Y, \eta)$ is a simple translation covering. Since $\Gamma_Y$ is a half-tree on $\mathrm{rk}(\M) + \mathrm{rel}(\M)$ vertices with an extra half-edge if and only if $\mathrm{rel}(\M) = 0$, it follows that $(Y, \eta)$ belongs to $\hyp(2r-2)$ when $\rel(\M) = 0$ and to $\hyp(r-1,r-1)$ when $\rel(\M) = 1$. Moreover the map $f$ is branched over zeros and commutes with the hyperelliptic involution by construction. 

It remains to show that $(X, \omega)$ is generic. Let $\N$ be the orbit closure of $(X, \omega)$. Notice that by construction two horizontal cylinders on $(X, \omega)$ have a ratio of lengths of core curves that is algebraic if and only if they are $\M$-equivalent. By Wright~\cite{Wcyl} Lemma 4.7, if two cylinders are $\N$-equivalent then they have an algebraic ratio of lengths of their core curves. Therefore, there are $\mathrm{rank}(\M) + \mathrm{rel}(\M)$ $\N$-equivalence classes of horizontal cylinders on $(X, \omega)$. Since $\N$ is contained in $\M$ its rank is bounded above by the rank of $\M$ and rel is bounded above by the rel of $\M$. Moreover, the dimension of the twist space of $(X, \omega)$ in $\N$ is bounded above by $\mathrm{rank}(\N) + \mathrm{rel}(\N)$.Combining these inequalities yields
\[ \mathrm{rank}(\N) + \mathrm{rel}(\N) \leq \mathrm{rank}(\M) + \mathrm{rel}(\M) \leq \mathrm{rank}(\N) + \mathrm{rel}(\N) \]
\noindent Therefore $\mathrm{rank}(\N) = \mathrm{rank}(\M)$ and $\mathrm{rel}(\N) = \mathrm{rel}(\M)$ and hence $\M$ and $\N$ coincide as affine invariant submanifolds. The exact same reasoning implies that $(Y, \eta)$ is generic in the component of the stratum of abelian differentials to which it belongs.
 \end{proof}

\bibliography{mybib}{}
\bibliographystyle{amsalpha}
\end{document}